\documentclass[11pt]{article}
\usepackage{amsfonts,latexsym,amssymb,amsthm,amsmath,graphicx,cases,mathrsfs}
\usepackage[ulem=normalem,draft]{changes}
\usepackage{paralist}
\usepackage{graphics} 
\usepackage{epsfig}
\usepackage{epstopdf}
\usepackage{ulem}
\usepackage[titletoc,title]{appendix}
\usepackage{empheq}
\usepackage{cancel}
\usepackage{tikz}
\usetikzlibrary{arrows,shapes}
\usetikzlibrary{decorations.pathmorphing,decorations.pathreplacing}
\usetikzlibrary{calc,patterns,angles,quotes}
\usepackage{todonotes}

\usepackage[colorlinks=true]{hyperref}
\hypersetup{urlcolor=blue, citecolor=red}
\allowdisplaybreaks
\newtheorem{theorem}{Theorem}[section]

\newtheorem{assumption}{Assumption}
\newtheorem{consequence}{Consequence}
\newtheorem{thm}{Theorem}

\newtheorem{lemma}[theorem]{Lemma}
\newtheorem{prop}{Proposition}
\newtheorem{clr}[theorem]{Corollary}

\theoremstyle{definition}

\newtheorem{rmk}{Remark}

\renewcommand\theassumption{H\arabic{assumption}}

\newcommand{\be}{\begin{equation}}
\newcommand{\ee}{\end{equation}}
\newcommand{\bsubeq}{\begin{subequations}}
	\newcommand{\esubeq}{\end{subequations}}

\newcommand{\ds}{\displaystyle}

\newcommand{\calL}{{\mathcal{L}}}
\newcommand{\calH}{{\mathcal{H}}}

\newcommand{\calN}{{\mathcal{N}}}
\newcommand{\calI}{{\mathcal{I}}}
\newcommand{\calX}{{\mathcal{X}}}
\newcommand{\calE}{{\mathcal{E}}}
\newcommand{\calF}{{\mathcal{F}}}

\newcommand{\calD}{{\mathcal{D}}}

\newcommand{\calA}{{\mathcal{A}}}
\newcommand{\calZ}{{\mathcal{Z}}}

\newcommand{\calU}{{\mathcal{U}}}

\newcommand{\calV}{{\mathcal{V}}}

\newcommand{\BR}{\mathbb{R}}

\newcommand{\BN}{\mathbb{N}}

\newcommand{\wti}{\widetilde}

\newcommand{\mb}[1]{\mathbf{#1}}

\newcommand{\bpm}{\begin{pmatrix}}
	\newcommand{\epm}{\end{pmatrix}}

\newcommand{\bbm}{\begin{bmatrix}}
	\newcommand{\ebm}{\end{bmatrix}}

\newcommand{\bem}{\begin{matrix}}
	\newcommand{\eem}{\end{matrix}}

\numberwithin{equation}{section}
\numberwithin{thm}{section}
\numberwithin{rmk}{section}
\numberwithin{prop}{section}

\newcommand{\bs}[1]{\boldsymbol{#1}}

\newcommand\rfrac[2]{{}^{#1}\!/_{#2}}

\newcommand{\norm}[1]{\left\lVert#1\right\rVert}
\newcommand{\abs}[1]{\left\lvert#1\right\rvert}

\newcommand{\ip}[2]{\langle #1, #2 \rangle}
\newcommand{\ipp}[2]{( #1, #2 )}
\newcommand{\Uad}{U_{ad}}

\newcommand{\nin}{\noindent}
\newcommand{\V}{\mathcal{V}}
\newcommand{\calP}{\mathcal{P}}

\newcommand{\Wtxy}{W(0,T;X,Y)}

\newcommand{\Wtay}{W(0,T;\calD(A),Y)}
\newcommand{\Wixy}{W_{\infty}(X,Y)}
\newcommand{\Wiv}{W_{\infty}(V,V')}
\newcommand{\Wiay}{W_{\infty}(\calD(A),Y)}

\newcommand{\Wiya}{W_{\infty}(Y,[\calD(A)]')}
\newcommand{\WOiay}{W^0_{\infty}(\calD(A),Y)}
\newcommand{\Liy}{L^2(I;Y)}
\newcommand{\LiV}{L^2(I;V)}

\newcommand{\LiVp}{L^2(I;V')}
\newcommand{\LiA}{L^2(I;\calD(A))}
\newcommand{\LiAd}{L^2(I;[\calD(A)]')}

\newcommand{\ol}[1]{\overline{#1}}

\newcommand{\ba}{_{(\bs \beta)}}

\long\def\/*#1*/{}

\newcommand{\tcr}[1]{\textcolor{red}{#1}}

\usepackage{todonotes}

\begin{document}
	\title{Differentiability of the Value Function on $H^1(\Omega)$ of Semilinear Parabolic Infinite Time Horizon Optimal Control Problems under Control Constraints}
	
	\author{Karl Kunisch \thanks{Institute for Mathematics and Scientific Computing, University of Graz, Heinrichstrasse 36, A-8010 Graz, Austria, and Radon Institute, Austrian Academy of Science, Linz, Austria. (karl.kunisch@uni-graz.at).}
		\and Buddhika Priyasad \thanks{Department of Mathematics and Statistics, University of Konstanz, 78457 Konstanz, Germany. (priyasad@uni-konstanz.de).}}
	
	\date{\today}
	\maketitle
	
	\begin{abstract}
		\nin An abstract framework guaranteeing the  continuous differentiability of local  value functions on $H^1(\Omega)$ associated with  optimal stabilization problems subject to  abstract semilinear parabolic equations  in the presence of norm constraints on the control is established. It guarantees the local well-posedness of  the associated Hamilton-Jacobi-Bellman equation in the classical sense. Examples illustrate that the assumptions  imposed on the dynamical system are satisfied for practically relevant semilinear equations. 
	\end{abstract}
	
	\section{Introduction.}
	Continuous differentiability of local value functions with respect to initial data in $H^1(\Omega)$ associated with norm-constrained optimal control problems on the infinite time horizon for  semilinear parabolic equations is investigated. This is an extension of our work in \cite{BK:2019} where a similar problem was investigated in the situation where the domain of the value function is $L^2(\Omega)$ rather than $H^1(\Omega)$. The motivation for these two different settings is the following. The class of semilinear equations for which local differentiability of the value function can be established  is  wider for $H^1(\Omega)$ initial data than  that for $L^2(\Omega)$ initial data. The $L^2(\Omega)$ framework, on the other hand, is of independent value. It can be  more flexible for consistent numerical realizations, for instance,    than the setting in   $H^1(\Omega)$.  Concerning the structure of proofs we can frequently proceed similarly as in \cite{BK:2019}, on a technical level, however,  many differences need to be overcome.
	
	To accomplish our goal we utilize  techniques for sensitivity  analysis of abstract infinite dimensional optimization problems. We consider the optimality conditions for our constraint optimal control problem  as parameter dependent  generalized equations and  apply  known results on the Lipschitz continuous dependence of their solutions \cite{D:1995}. Once Lipschitz continuity of the optimal controls with respect to initial data is established, differentiability of the value function and the  derivation of the Hamilton Jacobi Bellman (HJB)  equation will follow by  considerations which are rather straightforward by now. On a technical level, the most severe difficulties arise due to the fact that we consider infinite horizon rather the finite horizon  problems. These are the natural settings for optimal stabilization problems. We shall assume that the linear part of the  dynamical system is feedback stabilizable, see eg. \cite{KR:2019, RT:1975}.  Subsequently we  derive appropriate conditions which guarantee that the nonlinear equations allow stabilizing controls. These conditions will also guarantee that the adjoint state associated to a local optimal control is unique. Utilizing the transversality condition this later property is typically  immediate for finite horizon problems, but it is not at all obvious in the   infinite horizon case. Concerning the treatment of the norm constraints on the controls, utilizing properties of the adjoint states, we shall be able to argue that the constraints are inactive for the optimal controls beyond a certain time horizon. This property is essential (at least for our approach) to guarantee the uniqueness of the adjoint states and their Lipschitz continuous dependence on the parameters, see eg. \textit{Step} 2 of the proof of Proposition \ref{prop-p-W}.
	
	Let us point out a remarkable consequence of our analysis: While differentiability of the local value function $V$ on a subset of $H^1(\Omega)$  guarantees that its  Riesz representations lie in a subset of $H^1(\Omega)'$, the use of dynamic programming  will  guarantee that the Risez representations themselves are elements of  $H^1(\Omega)$. This will allow that the HJB equation can in fact be interpreted on subsets of $H^1(\Omega)$, which would otherwise not be possible, due to the appearance of the action of the state equation.
	
	The results of the paper require a smallness assumption on the initial data. While this is restrictive we conjecture that this is inevitable. Under structural assumptions, on the nonlinearity, for instance monotonicity,  it may be possible to obtain semi-global or global results, in the sense that for every  local optimal  solution satisfying a second order condition, the value function is differentiable in a neighborhood. This could be of interest for future work.
	
	Concerning the literature, there appears to be very little focusing on the $C^1$ property of the value function for infinite dimensional systems. For the finite dimensional case we can mention \cite[Chapter 5]{F:2001}, \cite{CF:2013}, \cite{Goe:2005}. It should  be recalled that  regularity of the value function is a special case,  since in general we can only expect its Lipschitz continuity, see \cite{BCD:1997}. On the other hand there are many papers on the sensitivity analysis of finite horizon optimal control problems with pointwise, and thus affine, control constraints. We  quote some of them \cite{BM:1999}, \cite{Gri:2004}, \cite{GV:2006}, \cite{GHH:2005}, \cite{Mal:2002}, \cite{MT:1999}, \cite{Tro:2000}, \cite{Wac:2005}. These papers are not written with the intention of application towards the HJB equation. The \cite{BKP:2018}, \cite{BKP:2019} Taylor functions expansions are provided for optimal stabilization problems without norm constraints related to bilinear problems and the Navier Stokes equation, respectively.
	
	The following sections are structured as follows. Section 2 contains the precise problem statement, the assumptions which are postulated throughout the remainder of the paper, and the statement of the two main theorems. Section 3 contains the technical developments which lead up to the analysis of the adjoint equations, the transversality condition, and second order optimality. Lipschitz continuous dependence of the optimal controls, and the associated states and adjoint states with respect to the initial condition is contained in Section 4. This is proved by verifying the  Dontchev-Robinson strong regular point  condition.
	The local $C^1$-property of the value function and the HJB equation are treated Section 5. Section 6 contains  concrete problems, which illustrate the applicability of the assumptions with respect to specific nonlinearities.

	\section{Differentiability of value function for optimal stabilization}\label{Sec-Diff}
	
	\subsection{Problem formulation}
	\nin Let $\Omega$ be an open connected bounded subset of $\BR^d$ with a $C^{1,1}$  boundary $\Gamma$.  The associated space-time cylinder is denoted by $Q = \Omega \times (0,\infty)$ and the associated lateral boundary by  $\Sigma = \Gamma \times (0,\infty)$. We consider the following the optimal stabilization problem in abstract form with associated value function,
	\begin{equation}\tag{$P$}\label{P}
	\V(y_0) = \inf_{\bem y \in \Wiay \\ u(t) \in \mathcal{U}_{ad}  \eem} \frac{1}{2} \int_{0}^{\infty} \norm{y(t)}^2_Y dt+ \frac{\alpha}{2} \int_{0}^{\infty} \norm{u(t)}^2_{\calU}dt,
	\end{equation}
	\nin subject to the semilinear parabolic equation
	\begin{subequations}\label{eq_y}
		\begin{empheq}[left=\empheqlbrace]{align}
		y_t &=  Ay + F(y) + Bu &\text{ in } L^2(I;Y),\\
		y(0) &= y_0  &\text{ in } V.
		\end{empheq}
	\end{subequations}
	Here $I=(0,\infty)$, $Y = L^2(\Omega),\, V = H^1(\Omega)$, and
	$\Wiay$ denotes the classical space for strong solutions   of semilinear parabolic  problems. Together with the operator $A$ it will be introduced below. Further $\mathcal U_{ad} \subset \mathcal{U}$ stands for the set of admissible controls which is defined as
	$\calU_{ad} = \{v \in \calU: \| v \|_{\calU}\le \eta\}$, where $\eta>0$,  and $\mathcal U$ is a Hilbert space which will be identified with its dual.  By $\ds \mathbb{P}_{\calU_{ad}}$ we denote the Hilbert space  projection of $\calU$ onto $\calU_{ad}$,  $U=L^2(I;\calU)$, and
	\begin{equation}\label{eq:Uad}
	U_{ad} = \{u \in U: \|u(t)\|_{\calU}\le \eta, \text{ for } a.e. \ t>0 \},
	\end{equation}
	and finally $B\in \mathcal{L}(\calU,Y)$. For this choice of admissible controls the dynamical system can be stabilized for all sufficiently small initial data in $V$, see  Corollary \ref{cor3.1} and Remark \ref{rem4.1}, provided that the pair $(A,B)$ is exponentially stabilizable.\\
	
	\nin Throughout $F$ stands for  the substitution operator associated to a mapping  $\mathfrak{f}: \mathbb{R} \to \mathbb{R}$, so that $(F y)(t) = \mathfrak{f} (y(t))$. Sufficient conditions  which guarantee the existence of  solutions to \eqref{eq_y} as well as solutions $(\bar{y}, \bar{u})$ to \eqref{P}, with $y_0\in V$ sufficiently small, will be given below. We shall also make use of the adjoint equation associated to an optimal state $\bar y$,  given by
	\begin{subequations}
		\begin{align}
		-p_t - A' p - F'(\bar{y})^* p &= -\bar{y} &\text{ in  } \LiAd,\\
		\lim_{t \rightarrow \infty} p(t) &= 0 &\text{ in } V'.
		\end{align}
	\end{subequations}
	Its solution $p$ will be considered in $L^2(I;Y)$ as well as  in $\Wiv = \{v\in L^2(I;V): v_t \in L^2(I;V')\}$, and eventually in $\Wiay$.

	\subsection{Further notation and assumptions on $A$}
	
	Since  $V$ is continuously and densely embedded in $Y$ the  inclusions  $V \subset Y \subset V'$ constitute a Gelfand triple with $Y$ as pivot space. Let $a$ be a continuous bilinear form on $V$ which is $V-Y$ coercive:
	\begin{equation}\label{eq:kk2}
	\exists \rho \in \mathbb{R} \text { and } \exists \theta>0 : \quad	a(v,v) + \rho |v|_Y^2 \geq \theta \norm{v}^2_V, \text{ for all } v \in V.
	\end{equation}
	This bilinear form induces the operator $A$ by means of
	\begin{equation*}
	\begin{array}l
	\calD(A)=\{v\in V:w\to a(v,w) \text { is } Y-\text{continuous}\}\\
	(Av,w)=-a(v,w), \forall v\in\calD(A), \forall w\in V.
	\end{array}
	\end{equation*}
	This operator is closed and densely defined in $Y$, and it can be uniquely extended to an operator $A\in \calL(V,V')$. It generates an analytic, exponentially bounded semigroup $e^{At}$ with $\|e^{At}\|_{\calL(Y)} \le e^{\rho t}$,  see eg. \cite[Part II, Chapter 1, pg 115]{BPDM:2007}. The adjoint of $A$, considered as operator in $Y$  will be denoted by  $A^*$. We assume that $ \calD(A)=\calD(A^*)$, algebraically and  topologically. We shall also  consider $A$ as a bounded linear operator in $\calL(\calD(A), Y)$, with dual $A' \in \calL(Y, [\calD(A)]')$.  Since $\calD(A^*)$ is dense in $Y$, we recall from  eg. \cite[Section 0.3, pg 6]{LT:2000}  that $A'$ is the unique extension of the operator $A^* \in \calL(\calD(A^*), Y)$ to an element in $\calL(Y, [\calD(A)]')$.
	%
	%
	\nin For any $T \in (0,\infty)$, we define the space $\Wtay$ which we endow with the norm
	\begin{equation}
	\norm{y}_{\Wtay} := \left( \norm{Ay}^2_{L^2(0,T;Y)}+ \norm{y}_{L^2(0,T;Y)} + \norm{\frac{dy}{dt}}^2_{L^2(0,T;Y)} \right)^{\rfrac{1}{2}}, \ y \in \Wtay.
	\end{equation}
	\nin Generally, given $T > 0$ and two Hilbert spaces $X \subset Y$, by $\Wtxy$ we denote the space
	\begin{equation}
	\Wtxy = \left\{ y \in L^2(0,T;X); \ \frac{dy}{dt} \in L^2(0,T;Y) \right\}.
	\end{equation}
	\nin For $T = \infty$, we write $\Wixy$ and $I = (0,\infty)$. We further define $W^0_{\infty}=\{y\in W_{\infty}: y(0)=0 \}$. We also set
	\begin{equation*}
	W(T,\infty;X,Y) = \bigg\{ y \in L^2(T,\infty;X); \ \frac{dy}{dt} \in L^2(T,\infty;Y) \bigg\}.
	\end{equation*}	
	We shall frequently use that $\Wiay$ embeds continuously into $C([0,\infty),V)$, see e.g. \cite[Theorem 4.2, Chapter 1]{LM:1972} and that $\ds \lim_{t \to \infty} y(t) = 0$ in $V$, for $y \in \Wiay$, see e.g. \cite{CK:2017}.
	Further $\|\calI\|$ denotes the norm of the embedding constant of $\Wiay$ into $\Liy$  and $\|i\|$ is the norm of the embedding $V$ into $Y$.\\
	
	\nin For $\delta > 0$ and $\bar y \in V$, we define the open neighborhoods $B_V(\delta) = \left\{ y \in V: \norm{y}_V < \delta \right\},$ and  $\quad B_V(\bar y, \delta) = \left\{ y \in V: \norm{y - \bar y}_V < \delta \right\}$.
	
	\subsection{Assumptions for problem \eqref{P}}\label{assum}
	
	Here we summarize the main assumptions which will be utilized {throughout the remainder of this paper}.
	
	\setcounter{assumption}{0}
	\renewcommand{\theassumption}{{A}\arabic{assumption}}
	\begin{assumption}\label{asp_1}
		The linear system $(A,B)$ with $B \in \calL(\calU,Y)$ is exponentially stabilizable , i.e. there exists $K \in \calL(Y,\cal{U})$ such that the semigroup $e^{(A-BK)t}$ is exponentially stable on $Y$.
	\end{assumption}
	\nin Regarding  assumption \eqref{asp_1}, we refer to e.g.  \cite{RT:1975}.
	
	\begin{assumption}\label{asp_2}
		The nonlinearity $\ds F: \Wiay \to L^2(I; Y)$ is twice continuously Fr\'{e}chet differentiable, with  second Fr\'{e}chet derivative $F''$ bounded on bounded subsets of $\Wiay$,  and $F(0) = 0$.
	\end{assumption}
	
	\begin{assumption}\label{asp_3}
		$\ds F: \Wtay \to L^1(0,T;\calH')$ is weak-to-weak continuous for every $T>0$, for some Hilbert space $\calH$ which embeds continuously and densely in $Y$.
	\end{assumption}
	\nin Recall that $\ds \left( L^1(0,T; \calH') \right)'  = L^{\infty}(0,T;\calH)$, see \cite[Theorem 7.1.23(iv), p 164]{EE:2004}. Moreover, $L^{\infty}(0,T; \calH)$ is dense in $L^2(0,T;Y)$, see \cite[Lemma A.1, p 2231]{MS:2017}.
	\begin{assumption}\label{asp_4}
		$F'({y})$ and  $F'({y})^*$ are elements of $\calL{(L^2(I;V), L^2(I;Y))}$   for each $y \in \Wiay$.	
	\end{assumption}

	\begin{rmk}\label{rmk-fprime}
		The requirement that $F(0) = 0$  in \eqref{asp_1} is consistent with
		the fact that we focus on the stabilization problem with $0$ as steady state for \eqref{def_eyu}. Without loss of generality we further assume that
		\begin{equation}\label{eq:fprime}
		F'(0) = 0,
		\end{equation}
		which can always be achieved by making $F'(0)$ to be perturbation of $A$.
	\end{rmk}
	
	\begin{rmk}
		{  Assumption \eqref{asp_2} is in only needed locally in the neighborhood of local solutions of  \eqref{P} which will be  under consideration. But it is convenient to assume this regularity globally.    Assumption \eqref{asp_4} will only be utilized at the $y-$component of local solutions $(\bar y, \bar u)$ of problem \eqref{P}. Such local solutions  $\bar y$ may enjoy higher regularity than being generic elements in $\Wiay $.
		}
	\end{rmk}
	
	
	\subsection{Main Theorems}
	
	\nin Now we present the main results of this paper. We shall refer to  value functions associated to local minima as `local value function'. The first theorem asserts  continuous differentiability of  local value functions $\V$ w.r.t. $y_0$, for all $y_0$ small enough. The second theorem establishes that $\V$ satisfies the HJB equation in the classical sense. This will be proven in Sections \ref{sec-lip}  and \ref{sec-HJB} below. Moreover we need to establish  the underlying assumption that problem \eqref{P} is well-posed. This will lead to a smallness assumption on the initial states $y_0$. \\
	
	\nin We shall further prove the Lipschitz  continuity  of the state, the adjoint state, and the control with respect to the initial condition $y_0 \in V$ in the neighborhood of  a locally optimal solution $(\bar y, \bar u)$ corresponding to a sufficiently small reference initial state $\bar y_0$. This will imply the desired  differentiability of the local value function associated to local minima.

	\begin{thm}\label{thm-CD}
		Let \eqref{asp_1}-\eqref{asp_4} hold. Then associated to each local solution $(\bar y(y_0), \bar u(y_0))$ of \eqref{P} there exists a neighborhood of $U(y_0)$ such that each local value function   $\calV: U(y_0) \subset V \to \BR$ is continuously differentiable, provided that $y_0$ is sufficiently close to the origin in $V$. Moreover the Riesz representor of $\calV'$ is an element of $C(U(y_0), V)$.
	\end{thm}
	

	\begin{thm}\label{thm-HJB}
		Let \eqref{asp_1}-\eqref{asp_4}  hold, and let $(\bar y(y_0), \bar u(y_0))$
		denote a local solution  of \eqref{P}  for $y_0$ with  sufficiently small  norm in $V$. Assume  that for some $T_0 >0$ we have $F ( y(\hat y_0,u)) \in C([0,T_0);V')$ where  $y(\hat y_0,u)$ denotes the solution to \eqref{eq_y} on $[0,T_0)$ with arbitrary  $u\in C([0,T_0);Y)$ and sufficiently small $\hat y_0$. Then the following Hamilton-Jacobi-Bellman equation holds in a neighborhood of $y_0$:
		\begin{equation}
		\V'(y)(A y + F(y)) + \frac{1}{2} \norm{y}^2_Y + \frac{\alpha}{2} \norm{ \mathbb{P}_{\mathcal{U}_{ad}} \left(-\frac{1}{\alpha}B^*\V'(y) \right)}^2_Y + \left( B^* \V'(y),\mathbb{P}_{\mathcal{U}_{ad}} \left(-\frac{1}{\alpha}B^*\V'(y) \right) \right)_Y = 0.
		\end{equation}
		Moreover the optimal feedback law is given by
		\begin{equation}
		{ u} =  \mathbb{P}_{\mathcal{U}_{ad}} \left(-\frac{1}{\alpha}  B^*\V'({ y}) \right).
		\end{equation}
	\end{thm}
	\nin The condition on the smallness of $y_0$ will be discussed in Remark \ref{rmk-pur} below. Roughly it involves well-posedness of the optimality system and second order sufficient optimality at local solutions. A more detailed statement of these two theorems will be given in Theorem  \ref{thm-CD-r} and Theorem  \ref{thm-HJB-r} below. 
	
	\section{Well-posedness of  \eqref{P} and optimality conditions}\label{sec_3}
	
	\subsection{Well-posedness of \eqref{P}}\label{Sec-P-WP}
	
	\nin Here we prove  well-posedness for \eqref{P} with small initial data. We recall the following consequence of the fact that  under our general assumptions $A$ is the generator of an analytic semigroup.
	
	
	\begin{consequence}\label{cons:2}
		For all $y_0 \in V, f \in L^2(0,T; Y)$, and $T > 0$, there exists a unique solution $y \in \Wtay$ to
		\begin{equation}\label{eq3.2}
		\dot{y} = A y + f, \quad y(0) = y_0.
		\end{equation}
		\nin Furthermore, $y$ satisfies
		\begin{equation}
		\norm{y}_{\Wtay} \leq c(T) \left( \norm{y_0}_{V} + \norm{f}_{L^2(0,T; Y)} \right)
		\end{equation}
		for a continuous function $c$. Assuming that $y \in L^2(I; Y)$, consider the equation
		\begin{equation*}
		\dot{y} = \underbrace{(A - \rho I)y}_{A_{\rho}} + \underbrace{\rho y + f}_{f_{\rho}}, \quad y(0) = y_0,
		\end{equation*}
		\nin where $f_{\rho} \in L^2(I; Y)$. Then the  operator $A_{\rho}$ generates a strongly continuous  analytic semigroup on $Y$ which is exponentially stable, see \cite[Theorem II.1.3.1]{BPDM:2007}. It follows that $y \in \Wiay$, that there exists $M_{\rho}$ such that
		\begin{equation}
		\norm{y}_{\Wiay} \leq M_{\rho} \left( \norm{y_0}_{V} + \norm{f_{\rho}}_{L^2(I; Y)} \right),
		\end{equation}
		\nin and that $y$ is the unique solution to \eqref{eq3.2}  in $\Wiay$, see \cite[Section 2.2]{BKP:2019} .
	\end{consequence}
	
	\begin{lemma}\label{lemma-Lip}
		There exists a constant $C > 0$, such that for all $\delta \in (0,1]$ and for all $y_1$ and $y_2$ in $\Wiay$ with $\ds \norm{y_1}_{\Wiay} \leq \delta$ and $\ds \norm{y_2}_{\Wiay} \leq \delta$, it holds that
		\begin{equation}
		\norm{F(y_1) - F(y_2)}_{L^2(I; Y)} \leq \delta C \norm{y_1 - y_2}_{\Wiay}.
		\end{equation}
	\end{lemma}
	\begin{proof}
		Let $y_1, y_2$ be as in the statement of the lemma. Using \eqref{asp_2} and Remark \ref{rmk-fprime} we obtain the estimate
		\begin{align*}
		&\norm{F(y_1) - F(y_2)}_{L^2(I,Y)} \leq \int_0^1  \norm{F'(y_1+t(y_2-y_1)) - F'(0) }_{\calL(\Wiay, L^2(I,Y))}\, dt \, \|y_2-y_1\|_{\Wiay} \\
		&\le \int_0^1 \int_0^1 \norm{F''(s(y_1+t(y_2-y_1)))(ty_2+(1-t)y_1)}_{\calL(\Wiay, L^2(I,Y))}\, ds dt \, \|y_2-y_1\|_{\Wiay}.
		\end{align*}
		\nin Now the claim follows by Assumption \eqref{asp_2}.
	\end{proof}
	\begin{lemma}\label{exp-exi}
		Let $\ds A_s$ be the generator of an exponentially stable analytic semigroup $\ds e^{A_s t}$ on $Y$. Let $C$ denote the constant from Lemma \ref{lemma-Lip}. Then there exists a constant $M_s$ such that for all $y_0 \in V$ and $f \in L^2(I; Y)$ with
		\begin{equation*}
		\tilde \gamma : = \norm{y_0}_V + \norm{f}_{L^2(I; Y)} \leq \frac{1}{4CM^2_s}
		\end{equation*}
		the system
		\begin{equation}\label{sys-As}
		y_t = A_s y + F(y) + f, \quad y(0) = y_0
		\end{equation}
		has a unique solution $y \in \Wiay$, which  satisfies
		\begin{equation*}
		\norm{y}_{\Wiay} \leq 2 M_s \tilde \gamma.
		\end{equation*}
	\end{lemma}
	
	\nin Utilizing Lemma \ref{lemma-Lip}, this lemma can be verified in the same manner as \cite[Lemma 5, p 6]{BKP:2019}. In the following corollary we shall use Lemma \ref{exp-exi} with $A_s = A-BK$, and the constant corresponding to $M_s$ will be denoted by $M_K$. We also recall the constant $\eta$ from \eqref{eq:Uad}.
	
	\begin{clr} \label{cor3.1}
		For all $\ds y_0 \in Y$ with
		\begin{equation*}
		\norm{y_0}_V \leq \min \left\{ \frac{1}{4CM^2_K}, \frac{\eta}{2M_K\norm{K}_{\calL(Y)}\norm{\calI}} \right\}
		\end{equation*}
		there exists a control $u \in {\Uad}$ such that the system
		\begin{equation}\label{eq3.6}
		y_t = A y + F(y) + Bu, \quad y(0) = y_0
		\end{equation}
		has a unique solution $y \in W_{\infty}$ satisfying
		\begin{equation}\label{eq3.6a}
		\begin{aligned}
		\norm{y}_{\Wiay} &\leq 2 M_K \norm{y_0}_V, \text{ and } \\
		\norm{u}_{U} \leq \|K\|_{\calL(Y,\,\calU)} \norm{\calI} \norm{y}_{\Wiay} &\le 2M_K \norm{y_0}_V \norm{K}_{\calL(Y,\,\calU)}\| \calI \|.
		\end{aligned}		
		\end{equation}
	\end{clr}
	
	\begin{proof}
		By Assumption \eqref{asp_1}, there exists $K$ such that $A - BK$ generates an exponentially stable analytic semigroup on $Y$. Taking $u = -Ky$, equation \eqref{eq3.6} becomes
		\begin{equation}\label{eq:3.7}
		y_t = (A - BK) y + F(y), \quad y(0) = y_0.
		\end{equation}
		Then by Lemma \ref{exp-exi} with $\tilde \gamma = \norm{y_0}_V$ there exists $M_K$ such that \eqref{eq:3.7} has a solution  $y \in \Wiay$ satisfying
		\begin{equation*}
		\norm{y}_{\Wiay} \leq 2 M_K  \norm{y_0}_V,
		\end{equation*}
		and thus the first inequality in \eqref{eq3.6a} holds. For the feedback control we obtain
		\begin{equation}\label{eq3.9}
		\| u \|_U = \|Ky\|_U \le \|K\|_{\calL(Y,\,\calU)} \| y \|_{\tcr{\Liy}} \le \| K \|_{\calL(Y,\,\calU)} \| \calI \| \| y \|_{W_\infty} \le 2M_K \| y_0 \|_V \| K \|_{\calL(Y,\,\calU)} \| \calI \|,
		\end{equation}
		and thus the second inequality in \eqref{eq3.6a} holds. We still need  to assert  that $u \in \Uad$. This follows from the second smallness condition on $\norm{y_0}_V$ and \eqref{eq3.9}.
	\end{proof}	
	
	\begin{rmk} \label{rem4.1} In the above proof  stabilization was achieved by the feedback control $u = -Ky$. For this $u$ to be admissible it is needed that $\calU_{ad}$  has nonempty interior. The upper bound $\eta$
		could be allowed to be time dependent as long as it satisfies $\ds \inf_{t\ge 0}|\eta(t)|>0$.
	\end{rmk}
	
	\begin{clr}\label{clr-rho}
		Let $\ds y_0 \in V$ and let $u \in \Uad$ be such that the system
		\begin{equation}\label{eq:3.6}
		y_t = A y + F(y) + Bu, \quad y(0) = y_0
		\end{equation}
		has a unique solution $y \in L^2(I;V)$. If
		\begin{align*}
		\gamma : = \norm{y_0}_V + \norm{\rho y + Bu}_{L^2(I; Y)}
		\leq \min \left\{ \frac{1}{4CM^2_{\rho}}, \frac{\eta}{2M_{\rho}\norm{K}_{\calL(Y,U)}\norm{\calI}} \right\},
		\end{align*}
		then $y \in \Wiay$ and it holds that
		\begin{equation*}
		\norm{y}_{\Wiay} \leq 2 M_{\rho} \gamma.
		\end{equation*}
	\end{clr}	
	\begin{proof}
		\nin Since $y \in L^2(I; Y)$, we can apply Lemma \ref{exp-exi} to the equivalent system
		\begin{equation*}
		y_t = (A - \rho I)y + F(y) + \tilde{f},
		\end{equation*}
		\nin where $\tilde{f} = \rho y + Bu$. This proves the assertion.
	\end{proof}
	\begin{lemma}\label{WPP}
		There exists $\delta_1 > 0$ such that for all $\ds y_0 \in B_V(\delta_1)$, problem \eqref{P} possesses a solution $(\bar{y}, \bar{u}) \in \Wiay \times \Uad$. Moreover, there exists a constant $M > 0$ independent of $y_0$ such that
		\begin{equation}\label{eq:op_est}
		\max \big\{ \norm{\bar{y}}_{\Wiay}, \norm{\bar{u}}_U \big\} \leq M \norm{y_0}_{V}.
		\end{equation}
	\end{lemma}
	\begin{proof}
		\nin The proof of this lemma follows with analogous argumentation as provided in \cite[Lemma 8]{BKP:2019}. Let us choose, $\ds \delta_1 \leq \min \left\{ \frac{1}{4CM^2_K}, \frac{\eta}{2M_K\norm{K}_{\calL(Y,U)}\norm{\calI}} \right\}$, where $C$ as in Lemma  \ref{lemma-Lip} and $M_K$ denotes the constant from  Corollary \ref{cor3.1}. We obtain from Corollary \ref{cor3.1} that for each $y_0 \in B_V(\delta_1)$, there exists a control $u \in \Uad$ with associated state $y$ satisfying
		\begin{equation}\label{max_yu}
		\max \big\{ \norm{u}_{U}, \norm{y}_{\Wiay} \big\} \leq \tilde M \norm{y_0}_V,
		\end{equation}
		\nin where $\ds \tilde M = 2M_K \text{ max } \big(1, \norm{\calI} \norm{K}_{\calL(Y,\,\calU)} \big)$. We can thus consider a minimizing sequence $\ds (y_n, u_n)_{n \in \mathbb{N}} \in \Wiay \times \Uad$ with $\ds J(y_n, u_n) \leq \frac{1}{2} \tilde M^2 \norm{y_0}^2_Y(1 + \alpha)$. Consequently for all $n \in \mathbb{N}$ we have
		\begin{equation}\label{eq:aux1}
		\norm{y_n}_{L^2(I;Y)} \leq \tilde M \norm{y_0}_Y \sqrt{1 + \alpha} \quad \text{and} \quad \norm{u_n}_{L^2(I;\,\calU)} \leq \tilde M \norm{y_0}_Y \sqrt{\frac{1 + \alpha}{\alpha}}.
		\end{equation}
		\nin We set $\eta(\alpha,\tilde M) =  \|i\|\Big[ 1 + \tilde M\ \sqrt{(1 + \alpha)} \Big( \rho + \frac{\norm{B}_{\calL(\calU, Y)}}{\sqrt{\alpha}} \Big) \Big]$. Then we have $\norm{y_0}_V + \norm{\rho y_n + B u_n }_{L^2(I;Y)} \leq \eta(\alpha,\tilde M) \norm{y_0}_V$. After further reduction of $\delta_1$, we obtain for $M_{\rho}$ from Corollary \ref{clr-rho}:
		\begin{equation*}
		\gamma = \norm{y_0}_V + \norm{\rho y_n + B u_n }_{L^2(I;Y)} \leq \frac{1}{4CM^2_{\rho}}, \ \text{ if } y_0 \in B_V(\delta_1).
		\end{equation*}	
		It  follows from this corollary that the sequence $\{y_n\}_{n \in \mathbb{N}}$ is bounded in $\Wiay$ with
		\begin{equation}\label{eq:aux2}
		\sup_{n \in \mathbb{N}} \norm{y_n}_{\Wiay} \leq 2 M_\rho \eta(\alpha, \tilde M) \norm{y_0}_{V}.
		\end{equation}
		Extracting if necessary a subsequence, there exists $\ds (\bar{y}, \bar{u}) \in \Wiay \times U$ such that $\ds ({y}_n, {u}_n) \rightharpoonup (\bar{y}, \bar{u}) \in \Wiay \times U$, and $(\bar{y}, \bar{u})$ satisfies \eqref{max_yu}.\\
		
		\nin Let us prove that $(\bar{y}, \bar{u})$ is feasible and optimal.  Since $\Uad$ is weakly sequentially closed and $u_n \in \Uad$, we find that $\bar{u} \in \Uad$. For each  fixed $T > 0$  and arbitrary $z \in L^{\infty}(0,T;\calH) \subset L^2(0,T;Y)$, we have for all $\ds n \in \mathbb{N}$ that
		\begin{equation}\label{eq:aux3}
		\int_{0}^{T} \ip{\dot{y}_n(t)}{z(t)}_{Y} dt = \int_{0}^{T} \ip{A y_n (t) - F(y_n(t)) + Bu_n(t)}{z(t)}_{Y} dt.
		\end{equation}
		\nin Since $\dot{y}_n \rightharpoonup \dot{y}$ in $L^2(0,T; Y)$, we can pass to the limit in the l.h.s. of the above equality. Moreover, since $A y_n \rightharpoonup A y$ in $L^2(0,T; Y)$,
		\begin{equation*}
		\int_{0}^{T} \ip{A y_n (t)}{z(t)}_{Y} dt \xrightarrow[n \rightarrow \infty]{} \int_{0}^{T} \ip{A \bar{y}(t)}{z(t)}_{Y} dt.
		\end{equation*}
		\nin Analogously, we obtain that
		\begin{equation*}
		\int_{0}^{T} \ip{B u_n (t)}{z(t)}_{Y} dt \xrightarrow[n \rightarrow \infty]{} \int_{0}^{T} \ip{B \bar{u}(t)}{z(t)}_{Y} dt, \ \text{ for each } z \in L^2(0,T;Y).
		\end{equation*}
		\nin If moreover $z \in L^{\infty}(0,T;\calH) \subset L^2(0,T;Y)$, we assert by \eqref{asp_3} that
		\begin{align*}
		\int_{0}^{T} \ip{F (y_n(t)) - F(\bar{y}(t))}{z(t)}_{\calH',\calH} dt \xrightarrow[n \rightarrow \infty]{} 0.
		\end{align*}
		\nin Thus we have for all $z \in L^{\infty}(0,T;\calH)$
		\begin{equation}\label{eq:aux4}
		\int_{0}^{T} \ip{\dot{y}(t) - A y(t) - Bu(t)}{z(t)}_{Y} dt = \int_{0}^{T} \ip{F(y(t))}{z(t)}_{Y} dt.
		\end{equation}
		\nin Since $\ds \dot{y} - A y - Bu \in L^2(0,T;Y)$ and $L^{\infty}(0,T;\calH)$ is dense in $L^2(0,T;Y)$ we conclude that \eqref{eq:aux4} holds for all $z \in L^2(0,T;Y)$. Thus $(\bar{y},\bar{u})$  is feasible. By weak lower semicontinuitiy of norms it follows that $\ds J(\bar{y},\bar{u}) \leq \liminf_{n \rightarrow \infty} J(\bar{y}_n,\bar{u}_n)$, which proves the optimality of $(\bar{y},\bar{u})$, and \eqref{eq:op_est} follows from \eqref{eq:aux1}.
	\end{proof}
	
	\nin For the derivation of the optimality system for \eqref{P}, we need the following lemma which is taken from  \cite[Lemma 10]{BK:2019}.
	
	\begin{lemma}\label{lem-pur}
		Let $\ds G \in \calL(\Wiay,L^2(I;Y))$ such that $\ds \norm{G} < \frac{1}{M_K}$, where $\norm{G}$ denotes the operator norm of $G$. Then for all $\ds f \in L^2(I;Y)$ and $y_0 \in V$, there exists a unique solution to the problem:
		\begin{equation*}
		y_t = (A - BK)y(t) + (Gy)(t) + f(t), \quad y = y_0.
		\end{equation*}
		\nin Moreover,
		\begin{equation*}
		\norm{y}_{\Wiay} \leq \frac{M_K}{1 - {M_K \|G\|}} \left( \norm{f}_{L^2(I;Y)} + \norm{y_0}_V \right).
		\end{equation*}
	\end{lemma}	
	
	\subsection{Regular point condition and first order optimality condition}
	
	\nin To establish optimality conditions for \eqref{P}, we consider  \eqref{P} as a special case of the following abstract optimization problem 	
	\begin{equation}\label{Pq}
	\begin{cases}
	\min \ f(y,u),\\
	e(y,u;y_0) = 0, \quad y \in \Wiay,\ u \in \Uad.
	\end{cases}
	\end{equation}
	\nin for $y_0\in V$. Thus  $e$  describes a parametric equality constraint.
	Indeed, the relationship between these two problems is provided by
	\begin{equation}\label{eq:fyu}
	f(y,u) = \frac{1}{2}\int_{0}^{\infty} \norm{y(t)}^2_{Y} dt + \frac{\alpha}{2}\int_{0}^{\infty} \norm{u(t)}^2_{\calU} dt,
	\end{equation}
	and
	\begin{equation}\label{def_eyu}
	e(y,u,y_0) = \bpm y_t - A y - F(y) - Bu \\ y(0) - y_0, \epm
	\end{equation}
	with  $f:\Wiay \times \Uad \longrightarrow \BR^+$ and $ e:\Wiay \times \Uad \times V \longrightarrow \Liy \times V.$ In what follows,
	\begin{enumerate}[i.]
		\item  $ y_0 \in V$ denotes  a nominal reference parameter, and
		\item $(\bar y, \bar u)$ is a local solution (\hyperref[Pq]{$P_{\bar y_0}$}).
	\end{enumerate}		
	With Assumption \eqref{asp_1} -  \eqref{asp_3} holding it follows that 	
	\begin{enumerate}	
		\item [iii.] $f: \Wiay \times \Uad \longrightarrow \BR^+$ is twice continuously differentiable in a neighborhood of $(\bar y, \bar u)$,
		\item [iv.] $e: \Wiay \times \Uad \times V \longrightarrow \Liy \times V$ is continuous, and twice continuously differentiable w.r.t. $(y,u)$, with first and second derivative Lipschitz continuous in a neighborhood of $(\bar y, \bar u, y_0)$.
	\end{enumerate}	
	
	\nin We introduce the Lagrangian $\calL: \Wiay \times \Uad \times \Liy \times V' \longrightarrow \BR$ associated to \eqref{Pq} by
	\begin{equation}\label{La_fu}
	\calL(y,u, y_0, \lambda) = f(y,u) + \ipp{\lambda}{e(y,u, y_0)}_{\Liy \times V', \Liy \times V}.
	\end{equation}
	\nin Here the initial condition $y_0 \in V$ enters as an index.  We  say that the regular point condition is satisfied at $(\bar y, \bar u, y_0) \in \Wiay \times \Uad \times V$, if
	\begin{equation}\label{reg_pt}
	0 \in \text{int } \left\{ e'(\bar y, \bar u, y_0) \bpm \Wiay \\ \Uad - \bar u \epm  \right\},
	\end{equation}
	\nin where $int$ denotes the interior in the $\Liy \times V$ topology, and the prime stands for the derivative with respect to $(y,u)$. If this condition holds then the existence of a Lagrange multiplier $\lambda_0 \in \Liy \times V'$ is guaranteed such that the following first order optimality condition holds, see e.g. \cite{MZ:1979}:
	\begin{equation}
	\begin{cases}
	\calL_y(\bar y, \bar u, y_0, \lambda_0) = 0,\\
	\ipp{\calL_u(\bar y, \bar u, y_0, \lambda_0)}{u - \bar u}_{U} \geq 0, \ \forall u \in \Uad,\\
	e(\bar y, \bar u, y_0) = 0.
	\end{cases}
	\end{equation}
	\nin This is equivalent to
	\begin{equation}\label{eq:rp_h1}
	\begin{aligned}
	\begin{cases}
	\calL_y(\bar y, \bar u, y_0, \lambda_0) = 0,\\
	0 \in \calL_u(\bar y, \bar u, y_0, \lambda_0) + \partial \mb{I}_{\Uad}(\bar u),\\
	e(\bar y, \bar u, y_0) = 0,
	\end{cases}
	\end{aligned}	
	\end{equation}
	\nin where
	\begin{equation}
	\partial \mb{I}_{\Uad}(\bar u) = \left\{ \tilde{u} \in U : \ipp{\tilde{u}(t)}{v(t) - \bar u(t)}_{\calU} \leq 0, \forall t \in I, \ v \in U_{ad}  \right\},
	\end{equation}
	
	\nin In the next proposition  the regular point condition is expressed for our particular constraint $e=0$ and the first order optimality conditions for problem \eqref{P} are established.
	\begin{prop}\label{prop:adj}There exists $\delta_2 \in (0,\delta_1]$ such that each local solution $(\bar{y}, \bar{u})$ with $y_0 \in B_V(\delta_2)$ is a regular point, i.e. \eqref{reg_pt} is satisfied, and there exists an adjoint state $(\bar p, \bar p_1) \in L^2(I;Y) \times V'$ satisfying		
		\begin{align}
		\ipp{v_t - A v - F'(\bar{y})v}{\bar p}_{L^2(I;Y)} + \ip{v(0)}{\bar p_1}_{V,V'} + \ipp{\bar y}{v}_{L^2(I;Y)} &= 0, \text{ for all } v \in \Wiay, \label{adjP-1}\\
		\ip{\alpha \bar{u} - B^* \bar p}{u - \bar{u}}_{U} & \geq 0, \text{ for all } u \in \Uad. \label{adjP-3}
		\end{align}
		\nin Moreover $\bar p$ satisfies
		\begin{equation*}
		- \bar p_t - A' \bar p - F'(\bar{y})^* \bar p = -\bar{y} \quad \text{  in  } L^2(I;[\calD(A)]'),
		\end{equation*}	
		\begin{equation}
		\bar p \in \Wiya, \; \bar p_1=\bar p(0), \text{ and }\lim_{t \rightarrow \infty} \bar p(t) = 0. \quad \text{in } V' \label{adjP-2}
		\end{equation}
		In addition there exists $\wti{M} > 0$, independent of $y_0 \in B_V(\delta_2)$, such that
		\begin{equation}\label{adjP-4}
		\norm{\bar p}_{\Wiya} \leq \wti M \norm{y_0}_V.
		\end{equation}
	\end{prop}
	\begin{proof}
		To verify the regular point condition, we evaluate $e$ defined in \eqref{def_eyu} at $(\bar{y}, \bar{u}, y_0)$. To check the claim on the range of $e'(\bar{y}, \bar{u}, y_0)$ we consider for arbitrary $\ds (r,s) \in L^2(I,Y) \times V$ the equation
		\begin{equation}\label{eq:zeq}
		z_t - A z - F'(\bar{y})z - B(w - \bar{u}) = r, \ z(0) = s,
		\end{equation}
		\nin for unknowns $(z,w) \in \Wiay \times {\Uad}$. By taking $w = -Kz \in U$ we obtain
		\begin{equation}\label{eq:aux13}
		z_t - (A - BK)z - F'(\bar{y})z + B \bar{u} = r, \ z(0) = s. \nonumber
		\end{equation}
		\nin {We apply Lemma \ref{lem-pur} to this equation with $\ds G = - F'(\bar{y})$ and $ f = r - B \bar{u}$. By Lemma \ref{WPP} and \eqref{eq:fprime} in Remark \ref{rmk-fprime} there exists $\delta_2 \in (0,\delta_1]$ such that $\|F'(\bar{y})\|_{\calL(\Wiay,L^2(I;Y))} \le \frac{1}{2} M_K$. Consequently by Lemma \ref{lem-pur}  there exists $\wti{M}$ such that}
		\begin{align}
		\norm{z}_{\Wiay} &\leq \wti{M} \big( \norm{r}_{L^2(I; Y)} + \norm{s}_{V} + \norm{B}_{ \calL(\calU, Y)} \norm{\bar{u}}_U\big) \nonumber \\
		&\leq \wti{M} \big( \norm{r}_{L^2(I; Y)} + \norm{s}_{V} + \norm{B}_{\calL(\calU, Y)} M \norm{y_0}_V \big) \label{eq:z_est},
		\end{align}
		\nin with $M$ as in \eqref{eq:op_est}. We still need to check whether $w = -Kz$ is feasible, which will be the case if $w(t) \leq \eta $ for a.e. $t \in I$. Indeed we have
		\begin{equation}\label{eq:w_est}
		\norm{w(t)}_Y \leq \norm{K}_{\calL(Y,\,\calU )} \norm{z(t)}_{Y} \leq \norm{K}_{\calL(Y,\,\calU)} \norm{\calI} \wti{M} \big( \norm{r}_{L^2(I; Y)} + \norm{s}_{V} + \norm{B}_{\calL(\calU, Y)} M \norm{y_0}_V \big). \nonumber
		\end{equation}
		\nin Consequently, possibly after further reducing $\delta_2$, and choosing $\tilde{\delta} > 0$ sufficiently small we have
		\begin{equation}
		\norm{w}_{L^{\infty}(I;Y)} \leq \eta \text{ for all } \ds y_0 \in B_Y(\delta_2) \text{ and all } (r,s) \text{ satisfying } \norm{(r,s)}_{L^2(I;Y) \times V} \leq \tilde{\delta}.
		\end{equation}
		\nin Consequently the regular point condition is satisfied. Hence there exists a multiplier $\ds \lambda = (\bar p, \bar p_1) \in L^2(I;Y) \times V'$ satisfying,
		\begin{equation}\label{eq:Lderi}
		\begin{aligned}
		\ip{\calL_y (\bar{y}, \bar{u}, y_0, \bar p, \bar p_1)}{v}_{\Wiay ', \Wiay} &= 0, \ \forall v \in \Wiay, \\
		\ip{\calL_u (\bar{y}, \bar{u}, y_0, \bar p, \bar p_1)}{u - \bar{u}}_{U} &\geq 0, \ \forall u \in \Uad,
		\end{aligned}
		\end{equation}
		\nin where
		\begin{equation*}
		\calL(y, u, y_0, p, p_1) = J(y,u) + \int_0^{\infty} \ipp{p}{y_t - A y - F(y) - Bu}_{Y}dt + \ip{p_1}{y(0) - y_0}_{V, V'}.
		\end{equation*}
		\nin This implies that \eqref{adjP-1} and \eqref{adjP-3} hold. \\
		
		
		\nin By  \eqref{asp_4},  we have $F'(\bar{y})^* \bar p \in L^2(I;[\calD(A)]')$. { Thus $-A'$} $\bar p - F'(\bar{y})^* \bar p + \bar{y} \in L^2(I;[\calD(A)]')$, and  \eqref{adjP-1} implies that $\bar p \in \Wiya$. Next we verify that $\ds \lim_{t \rightarrow \infty} p(t) = 0$ in $V'$. For this purpose, we consider $A_{\rho}^{-\rfrac{1}{2}}p$ where $A_{\rho} = (A - \rho I)$ is exponentially stable. Since $p \in \Wiya$, we have $\ds A_{\rho}^{-\rfrac{1}{2}}p \in W_{\infty}(V,V') \subset C(I;Y)$. Then by \cite{CK:2017}, we have $\ds \lim_{t \rightarrow \infty} A_{\rho}^{-\rfrac{1}{2}}p(t) = 0 $ in $Y$. This yields $\ds \lim_{t \rightarrow \infty} p(t) = 0$ in $V'$. Thus we derived
		\begin{equation*}
		- \bar p_t - A^* \bar p - F'(\bar{y})^* \bar p = -\bar{y} \ \text{in } L^2(I;[\calD(A)]') \text{ and } \lim_{t \rightarrow \infty} \bar p(t) = 0 \ \text{in } V',
		\end{equation*}
		and \eqref{adjP-1}-\eqref{adjP-2} hold. Testing the first identity in \eqref{eq:Lderi} with $v \in L^2(I;Y)$ we also have $\bar p_1 = \bar p(0) \in V'$, which is well-defined since $\bar p \in \Wiya \subset C(I;V')$.

		\nin It remains to estimate $\ds \bar p \in \Wiya$.		\nin Let $\ds r \in L^2(I;Y)$ with $\ds \norm{r}_{L^2(I;Y)} \leq \tilde{\delta}$, and consider
		\begin{equation}\label{eq:zeq_new}
		z_t - A z - F'(\bar{y})z - B (w-\bar{u}) = -r, \ z(0) = 0.
		\end{equation}
		\nin Arguing as in \eqref{eq:zeq}-\eqref{eq:z_est} there exists a solution to \eqref{eq:zeq_new} with $w = - Kz$ such that
		\begin{equation}\label{eq:z_est_new}
		\norm{z}_{\Wiay} \leq \wti{M} \big( \tilde{\delta} + \norm{B}_{\calL(\calU, Y)} M \norm{y_0}_V \big) \leq \wti{M} \big( \tilde{\delta} + \norm{B}_{\calL(\calU,  Y)} M \delta_2 \big) =: C_1.
		\end{equation}
		\nin  From \eqref{eq:w_est} we have that $\ds \norm{w}_{L^{\infty}(I,\,\calU )} \leq \eta$. Let us now observe that by \eqref{adjP-1} with $v = z, \ v(0) = z(0) = 0$,
		\begin{equation}\label{eq:kk8}
		\begin{aligned}
		\ipp{\bar p}{r}_{L^2(I,Y)} &= \ipp{\bar p}{-z_t + A z + F'(\bar{y})z}_{L^2(I,Y)} + \ipp{\bar p}{B(w-\bar{u})}_{L^2(I;Y)},\\
		&= \ipp{\bar y}{z}_{\Liy} + \ipp{B^* \bar p}{w - \bar{u}}_U.
		\end{aligned}
		\end{equation}		
		\nin We next estimate using \eqref{eq:z_est_new}
		\begin{align*}
		\ip{\bar p}{r}_{L^2(I;Y)} & \leq \norm{\bar{y}}_{L^2(I,Y)} \norm{z}_{L^2(I;Y)} + \alpha \ip{\bar{u}}{w - \bar{u}}_U
		&\leq \left( \norm{\bar{y}}_{L^2(I;Y)} + \alpha \norm{\bar{u}}_U \right) \left( \tilde C_1 + \eta + \norm{\bar{u}}_U \right),
		\end{align*}
		\nin where $\tilde C_1$ depends on $C_1$ and the embedding $\Wiay$ into $\Liy$. By \eqref{eq:op_est}, this implies the existence of a constant $C_2$ such that
		\begin{equation*}
		\sup_{\norm{r}_{L^2(I;Y)} \leq \tilde \delta} \ipp{\bar p}{r}_{L^2(I;Y)} \leq C_2 \norm{y_0}_V
		\end{equation*}
		and thus
		\begin{equation}\label{est:p}
		\norm{\bar p}_{L^2(I;Y)} \leq \frac {C_2}{\tilde \delta} \norm{y_0}_V, \quad \text{for all } y_0 \in B_V(\delta_2).
		\end{equation}
		\nin  We estimate, again using \eqref{asp_4}
		\begin{align*}
		\norm{\bar p_t}_{L^2(I;[\calD(A)]')} &\leq \norm{A^* \bar p + F'(\bar{y})^* \bar p - \bar{y}}_{L^2(I;[\calD(A)]')}
		\leq C_3\norm{\bar p}_{L^2(I;Y)} + C_4 (\norm{\bar p}_{L^2(I;Y)} + \norm{\bar{y}}_{L^2(I;Y)}).
		\end{align*}
		\nin By \eqref{eq:op_est} and \eqref{est:p} we obtain $\norm{\bar p_t} _{L^2(I;[\calD(A)]')} \leq C_5 \norm{y_0}_V$. Combining this estimate with \eqref{est:p}  yields \eqref{adjP-4}. 
	\end{proof}	
	
	In the following result we obtain stronger properties for the adjoint states $\bar p$.

	\begin{prop}\label{prop-p-W}
		
		For each local solution $(\bar{y}, \bar{u})$ with $y_0 \in B_V(\delta_2)$
		the associated adjoint state $\bar p$ is unique. It  satisfies $\bar p \in \Wiay$,
		\begin{equation}\label{eq:kk5}
		-\bar p_t - A^* \bar p - F'(\bar y)^* \bar p = -\bar y \quad \text{ in } L^2(I;Y)
		\end{equation}
		and
		\begin{equation}\label{eq:kk6}
		\lim_{t \rightarrow \infty} \bar{p}(t) = 0 \quad \text{ in } V.
		\end{equation}
		Moreover, there exists $\hat{M} > 0$, independent of $y_0 \in B_V(\delta_2)$, such that
		\begin{equation}\label{eq:kk7}
		\norm{\bar p}_{\Wiay} \leq \hat M \norm{y_0}_V, \text{ and } \bar u \in C(\ol{I}; \calU).
		\end{equation}
	\end{prop}
	
	\begin{proof}
		Throughout we fix an adjoint state $\bar p$ associated to a local solution $(\bar{y}, \bar{u})$ with $y_0 \in B_V(\delta_2)$.\\
		
		\nin \uline{\textit{Step 1}: ($W_{\infty}(V,V')$-regularity)}. Since $\bar p \in L^2(I;Y)$ there exists a monotonically increasing sequence $\ds \{ t_n \}_{n \in \BN}$ with $\ds \lim_{n\to\infty} t_n=\infty$ and  $\bar p_n = \bar p(t_n) \to 0 $ in $Y$. 
		Now consider following problem.
		\begin{equation}\label{eq_pV'}
		- p_t - A^* p - F'(\bar y)^* p = -\bar y \; \text{ in }  L^2(0,t_n;V'), \quad p(t_n) = \bar p_n  \text{ in } Y.
		\end{equation}
		Since  $F'({\bar y})^*\in \calL{(L^2(I;V), L^2(I;Y))}$ this problem admits a unique solution in  $W(0,t_n;V,V')$ which coincides with $\bar p$ on $[0,t_n]$. Using that $\ds \lim_{n\to\infty} t_n=\infty$ this implies that  $\bar p \in W_{loc}(0,\infty;V,V')$. Next we derive a bound for   $\bar p \in \Wiv$ .  By \eqref{eq_pV'} we obtain,
		\begin{equation*}
		-\frac{1}{2} \frac{d}{dt} \norm{\bar p(t)}^2_Y+ a(\bar p(t),\bar p(t)) +\rho \norm{\bar p(t)}_Y^2= \ip{F'(\bar y)^* \bar p(t)}{\bar p(t)}_{V',V} + \rho\norm{\bar p(t)}_Y^2 + \ipp{\bar y(t)}{\bar p(t)}_Y.
		\end{equation*}
		\nin By integrating w.r.t. $t$ on $(0,t_n)$  we obtain,	
		\begin{multline*}
		\frac{1}{2} \norm{\bar p(0)}^2_Y + \int^{t_n}_0 a(\bar p,\bar p)\,dt+ \rho \int_{0}^{t_n}  \|\bar p(t)\|_Y^2  \,dt \\
		\quad \le   \frac{1}{2} \norm{\bar p(t_n)}^2_Y + \int_{0}^{t_n} \|F'(\bar y)^*\bar p\|_{V'}\|\bar p\|_V dt + \int_{0}^{t_n}\left(\rho´+\frac{1}{2}\right)\|\bar p\|^2_Y \,dt + \frac{1}{2} \int_0^{t_n}\norm{\bar y}^2_Y dt,\\
		\quad \le   \frac{1}{2} \norm{\bar p(t_n)}^2_Y +  \frac{\theta}{2}  \int_{0}^{\infty} \|\bar p\|^2_{V} \,dt + \int_{0}^{t_n}\left(\frac{c_1^2}{2\theta} + \rho + \frac{1}{2} \right)\|\bar p\|^2_Y \,dt + \frac{1}{2} \int_0^{t_n}\norm{\bar y}_Y^2 dt,
		\end{multline*}	
		\nin where $c_1$ denotes the norm of $F'({y})\in \calL{(L^2(I;V), L^2(I;Y))}$ according to \eqref{asp_4}.
		Taking the limit $n\to \infty$ and using \eqref{eq:kk2} we obtain
		\begin{equation}\label{eq:kk12}
		{\theta} \int_{0}^{\infty} \norm{\bar p}^2_V dt \leq \left(\frac{c_1^2}{\theta} + 2\rho + 1 \right) \norm{\bar p}^2_{L^2(I;Y)} + 2 \norm{\bar y}^2_{L^2(I;Y)} < \infty.
		\end{equation}
		This estimate, together with \eqref{eq:op_est} and \eqref{adjP-4} implies the existence of a constant $c_2$ independent of $y_0\in  B_V(\delta_2)$ such that $\norm{\bar p}_{L^2(I;V)} \leq c_2 \norm{y_0}_V$.
		Combining this with
		\begin{equation}\label{eq_pV'-a}
		- \bar p_t - A^* \bar p - F'(\bar y)^* \bar p = -\bar y \quad \text{ in }  L^2(0,\infty;V'),
		\end{equation}
		we obtain $\bar p \in {W_\infty(V,V')}$ and   the existence of a constant $c_3$ such that
		\begin{equation} \label{eq:kk3}
		\norm{\bar p}_{W_\infty(V,V')} \leq c_3 \norm{y_0}_V, \forall y_0 \in B_V(\delta_2)
		\end{equation}
		follows. By \cite{CK:2017} this implies that $\ds \lim_{t \to \infty} \bar{p}(t) = 0 \text{ in } Y.$ The optimality condition \eqref{asp_3} implies that $\ds \bar u(t)= \frac{1}{\alpha} \mathbb{P}_{\calU_{ad}} (B^*\bar p(t))$, where $\ds \mathbb{P}_{\calU_{ad}}$ denotes the projection onto $\calU_{ad}$. Together with  $\bar p \in {W_\infty(V,V')} \subset C(I;Y)$ this implies that  $\bar u \in C(\ol{I}, \calU)$, which is the second claim in \eqref{eq:kk7}.\\
		
		\nin \uline{\textit{Step 2}: (Uniqueness of the multiplier).} Let $\bar p$ and $\bar q$ be two possibly different adjoint states and denote by $\delta \bar p= \bar q - \bar p$. We shall utilize the fact that there exists $T$ such that $ \|B^*\bar p(t)\|_{\calU} \le \eta$ and $ \|B^*\bar q(t)\|_{\calU} \le \eta$ for all $t \ge T$. Consequently  $\ds B^*\bar p(t) = \mathbb{P}_{\calU_{ad}} (B^*\bar p(t)) = \mathbb{P}_{\calU_{ad}} (B^*\bar q(t))=B^*\bar q(t) = \bar u(t)$ for all $t\ge T$. Let us now consider the construction utilized in  \eqref{eq:zeq_new}, now with $r\in S:= \{r\in L^2(T,\infty;Y):\|r\|_{L^2(T,\infty;Y)} \le \tilde\delta \} $. We construct function pairs $(z,w)$ with    $z\in W^0_\infty(T,\infty; \calD(A),Y)=\{z\in W_\infty(T,\infty; \calD(A),Y):z(T)=0\} $ and $w = -Kz$ with $\|w\|_{L^\infty(T,\infty; \calU)}\le \eta$  by means of
		\begin{equation*}
		z_t - A z - F'(\bar{y})z - B (w-\bar{u}) = -r, \ z(T) = 0,
		\end{equation*}
		for any $r\in S$. Note that $\|F'(\bar{y})\|_{\calL(W^0_\infty(T,\infty; \calD(A),Y),L^2(I;Y))}  \le \|F'(\bar{y})\|_{\calL(\Wiay,L^2(I;Y))} $. Consequently as in \eqref{eq:z_est_new} we obtain existence of a solution to the above equation with  $\|z\|_{W^0_\infty(T,\infty; \calD(A),Y)} \leq  C_1$, with $C_1$ independent of $r\in S$. Combining this with the equations satisfied by $\bar q$ and $\bar p$ we obtain
		for all $r\in S$, using that $z\in W^0_\infty(T,\infty; \calD(A),Y)$,
		\begin{equation*}
		\begin{array}{ll}
		\ipp{\delta p}{r}_{L^2(I,Y)} &= \ipp{\delta p}{
			-z_t + A z + F'(\bar{y})z}_{L^2(T,\infty;Y)} + \ipp{\bar p}{B(w-\bar{u})}_{L^2(T,\infty;Y)},\\
		&= \ipp{\bar y - \bar y}{z}_{L^2(T,\infty;Y)} + \ipp{B^* \delta p}{w - \bar{u}}_{L^2(T,\infty;\calU)} = \ipp{\bar u-\bar u}{w - \bar{u}}_{L^2(T,\infty;\calU)}=0.
		\end{array}
		\end{equation*}
		Here we used \eqref{adjP-3} and $\|B^*\bar p(t)\|_{\calU} \le \eta, \, \|B^*\bar q(t)\|_{\calU} \le \eta$ for all $t \ge T$ in an essential manner. The above equality implies that $\bar q= \bar p$ on $[0,\infty)$. Next we observe that $\bar q$ and $\bar p$ satisfy \eqref {eq_pV'-a} on $[0,T]$ with the same terminal value $\bar p(T)$ at $t=T$. Consequently $\bar q =\bar p$ on $[0,T]$ and the uniqueness of the adjoint state follows.\\
		
		%
		
		\nin \uline{\textit{Step 3}: ($W_{\infty}(\calD(A),Y)$-regularity).} The proof is very similar to that of \textit{Step} 1. Since $\bar p \in L^2(I;V)$   there exists a monotonically increasing sequence $\ds \{ t_n \}_{n \in \BN}$ with $\ds \lim_{n\to\infty} t_n=\infty$ and  $\bar p_n = \bar p(t_n) \to 0 $ in $V$. Now consider following problem.
		\begin{equation}\label{eq_pY}
		- p_t - A^* p - F'(\bar y)^* p = -\bar y \;\text{ in }  L^2(0,t_n;Y), \quad
		p(t_n) = \bar p_n  \text{ in } V.
		\end{equation}
		Since  $F'({\bar y})^*\in \calL{(L^2(I;V), L^2(I;Y))}$ this problem admits a unique solution in  $W(0,t_n;\calD(A),Y)$ which coincides with $\bar p$ on $[0,t_n]$. Using that $\ds \lim_{n\to\infty} t_n=\infty$ this implies that  $\bar p \in W_{loc}(0,\infty;\calD(A),Y)$. Next we obtain a bound for   $\bar p \in W(0,\infty;\calD(A),Y)$.  Taking the inner product in \eqref{eq_pY} with $-A^*_\rho \,{\bar p}= (-A^* +\rho I){\bar p}$  we obtain, 
		\begin{equation*}
		-\frac{1}{2} \frac{d}{dt} ( a(\bar p(t),\bar p(t))+  \rho \norm{\bar p(t)}^2_Y) +  \norm{A_\rho\bar p(t)}^2_Y  \le  \frac{1}{2}\norm{A_\rho\bar p(t)}^2_Y  +   \frac{3}{2} \left( \|F'(\bar y)^*\bar p(t)\|^2_Y + \rho \,\|\bar p(t) \|^2_Y + \|\bar y(t)\|^2_Y \right).
		\end{equation*}
		\nin By integrating w.r.t. $t$ on $(0,t_n)$  we obtain,	
		\begin{multline*}
		\frac{\theta}{2} \norm{\bar p(0)}^2_V + \frac{1}{2}  \int^{t_n}_0 \| A_\rho^*\bar p\|^2_Y\,dt \\ \le \frac{1}{2} a(\bar p(t_n),\bar p(t_n))+ \rho   \|\bar p(t_n)\|_Y^2
		+  \frac{3}{2} \int_0^{t_n} \left( \|F'(\bar y)^*\bar p(t)\|^2_Y + \rho \,\|\bar p(t) \|^2_Y + \|\bar y(t)\|^2_Y \right)\,dt.
		\end{multline*}
		Taking the limit $n\to  \infty$ implies that
		\begin{equation*}
		\theta \norm{\bar p(0)}^2_V +   \int^{\infty}_0 \| A_\rho^*\bar p\|^2_Y\,dt \; \le \;  3 \int_0^{\infty} \left( \|F'(\bar y)^*\bar p(t)\|^2_Y + \rho \,\|\bar p(t)\|^2_Y + \|\bar y(t)\|^2_Y \right)\,dt.
		\end{equation*}
		Thus by \eqref{eq:op_est}, \eqref{eq:kk3}, and \eqref{asp_4} there exists a constant $c_4$
		independent of $y_0\in  B_V(\delta_2)$ such that $\norm{\bar p}_{L^2(I;D(A))} \leq c_4 \norm{y_0}_V$.
		Combining this with \eqref{eq_pY}
		we obtain $\bar p \in {W_\infty(\calD(A),Y)}$ and   the existence of a constant $c_5$ such that
		\begin{equation*}
		\norm{\bar p}_{W_\infty(\calD(A),Y)} \leq c_5 \norm{y_0}_V, \forall y_0 \in B_V(\delta_2)
		\end{equation*}
		follows.  Finally $A_\rho\,\bar p \in W_\infty(I;Y,\calD(A'))$ and thus $\ds \lim_{t \to \infty} A_\rho\, \bar p(t) = 0$ in  $V'$  by \eqref{adjP-2}. This implies that  $\ds \lim_{t \to \infty}\bar p(t) = 0$ in  $V$.
		
	\end{proof}

	\subsection{Second order optimality condition}
	
	\nin Let $(\bar y, \bar u) \in \Wiay \times \Uad$ be a local solution to \eqref{P} with $y_0 \in B_V(\delta_2)$, so that the results of the previous section are available,  and let $\ds \calA \in \calL(\Wiay \times \Uad, \Liy \times \Uad)$ denote the operator representation of $\ds \calL''(\bar y, \bar u, y_0, \bar p, \bar p_1)$, i.e.
	\begin{equation}\label{def_A}
	\ipp{\calA (v_1,w_1)}{(v_2,w_2)}_{\Liy \times \Uad} = \calL''(\bar y, \bar u, y_0, \bar p, \bar p_1)((v_1,w_1),(v_2,w_2))
	\end{equation}
	\nin where $(v_i,w_i) \in \Wiay \times \Uad$ for $i = 1, 2$, and define
	\begin{equation}\label{def_E}
	\calE = e'(\bar y, \bar u, y_0) \in \calL(\Wiay \times \Uad, \Liy \times V).
	\end{equation}	
	Above again, the primes denote differentiation with respect to $(y,u)$.	
	
	We say that $\calA$ is positive definite on $\text{ker } \calE$  if
	\begin{equation}\label{snd_opt}
	\exists \kappa >  0: \ipp{\calA (v,w)}{(v,w)}_{\Liy \times \Uad} \geq \kappa \norm{(v,w)}^2_{\Wiay \times \Uad}, \quad \forall (v,w) \in \text{ker } \calE.
	\end{equation}

	\nin The following proposition derives the second order sufficient optimality conditions for  \eqref{P}. 
	\begin{prop}\label{le4.7}
		Consider problem \eqref{P} with \eqref{asp_1}-\eqref{asp_4}  holding. Then there exists $\ds \delta_3 \in (0, \delta_2]$ such that the second order sufficient optimality condition \eqref{snd_opt} is satisfied for \eqref{P} uniformly for all  local solutions with  $y_0 \in B_V(\delta_3)$.
	\end{prop}
	\begin{proof}
		\nin The second derivative of $e$ is given by
		\begin{equation}\label{sec_d_e}
		e''(\bar{y}, \bar{u}, y_0)((v_1,w_1), (v_2,w_2)) = \bpm F''(\bar{y}) (v_1, v_2) \\ 0 \epm, \quad \forall \ v_1, v_2 \in \Wiay, \ \forall w_1, w_2 \in U.
		\end{equation}
		\nin For the second derivative of $\calL$ w.r.t. $(y,u)$, we find
		\begin{multline}\label{sec_d_L.2}
		\calL''(\bar{y},\bar{u},y_0, \bar p, \bar p_1)((v_1,w_1),(v_2,w_2)) = \\ \int_{0}^{\infty} \ipp{v_1}{v_2}_Y dt + \alpha \int_{0}^{\infty} \ipp{w_1}{w_2}_Y dt  - \int_{0}^{\infty} \ipp{\bar p}{F''(\bar{y})(v_1, v_2)}_{Y} dt.
		\end{multline}
		\nin By \eqref{asp_2} for $F ''$ and Lemma \ref{WPP} , there exists $M_1$ such that for all $ v \in \Wiay, \nonumber$
		\begin{align}
		&\int_{0}^{\infty}  | \ipp{\bar p}{F''(\bar{y})(v, v)}_{Y} | dt \leq \int_{0}^{\infty} \norm{\bar p}_{Y} \norm{F''(\bar{y})(v, v)}_{Y}dt \ \\
		&\leq \norm{\bar p}_{\Liy} \norm{F''(\bar{y})(v, v)}_{\Liy}
		\leq M_1 \norm{\bar p}_{\Liy} \norm{v}^2_{\Wiay}, \nonumber
		\end{align}
		for each solution $(\bar{y}, \bar{u})$ of \eqref{P} with $y_0 \in B_V(\delta_2)$. Then we obtain
		\begin{multline}
		\calL''(\bar{y}, \bar{u}, y_0, \bar p, \bar p_1)((v,w),(v,w)) \geq \int_{0}^{\infty} \norm{v}_Y^2 dt + \alpha \int_{0}^{\infty} \norm{w}^2_{\calU} dt \\- \wti{M}_1 \norm{\bar p}_{\Liy} \norm{v}^2_{\Wiay}.
		\end{multline}
		Now let $0\neq (v,w) \in \ker \calE \subset \Wiay \times \Uad$, where $\calE$ as defined in \eqref{def_E} is evaluated at $(\bar y,\bar u)$. Then,
		\begin{equation*}
		v_t  - A v - F'(\bar{y})v - Bw = 0, \quad v(0) = 0.
		\end{equation*}
		Next choose $\rho > 0$, such that the semigroup generated by $(A - \rho I)$ is  exponentially stable. This is possible due to \eqref{asp_1}. We equivalently write the system in the previous equation as,
		\begin{equation*}
		v_t  - (A - \rho I) v - F'(\bar{y})v - \rho v - Bw = 0, \quad v(0) = 0.
		\end{equation*}
		\nin Now, we invoke Lemma \ref{lem-pur} with $A-BK$ replaced by $A - \rho I$, $\ds G = F'(\bar{y})$, and $f(t)=\rho v(t) + B w(t)$, and the role of the constant $M_K$ will now be assumed by a parameter $M_\rho$.  By selecting $\delta_3 \in (0, \delta_2]$ such that $\norm{\bar{y}}_{\Wiay}$ sufficiently small, we can guarantee that $\ds \norm{F'(\bar{y})}_{\calL(\Wiay; L^2(I;Y))} \leq \rfrac{1}{2 M_\rho}$, see \eqref{eq:op_est} and \eqref{eq:fprime} in Remark \ref{rmk-fprime}. Then the following estimate holds,
		\begin{equation*}
		\norm{v}_{\Wiay} \leq  2{M_\rho} \norm{v + Bw}_{L^2(I;Y)}.
		\end{equation*}
		This implies that
		\begin{equation}\label{ker_vw}
		\norm{v}^2_{\Wiay} \leq  \wti{M}_2 \left( \norm{v}^2_{L^2(I; Y)} + \norm{w}^2_{L^2(I; Y)} \right).
		\end{equation}
		for a constant  $\wti{M}_2$ depending on $M_\rho, \|B\|$. These preliminaries allow the following lower bound on  $\calL''$:
		\begin{align}
		\calL''(\bar{y}, \bar{u}, y_0, \bar p, \bar p_1)((v,w),(v,w)) &\geq \int_{0}^{\infty} \norm{v}_Y^2 dt + \alpha \int_{0}^{\infty} \norm{w}^2_Y dt - \wti{M}_1 \norm{\bar p}_{\Liy} \norm{v}^2_{\Wiay} \nonumber \\
		\text{by ~ \eqref{ker_vw} ~} \geq \int_{0}^{\infty} \norm{v}_Y^2 + \alpha & \int_{0}^{\infty} \norm{w}^2_Y - \wti{M}_1  \wti{M}_2 \norm{\bar p}_{\Liy} \left[ \norm{v}^2_{L^2(I; Y)} + \norm{w}^2_{L^2(I; Y)} \right] \nonumber\\
		= \left( 1 - \wti{M}_1 \wti{M}_2 \right. & \left. \norm{\bar p}_{\Liy}  \right) \norm{v}^2_{L^2(I;Y)} + \left( \alpha - \wti{M}_1  \wti{M}_2 \norm{\bar p}_{\Liy} \right) \norm{w}^2_{L^2(I;Y)} \nonumber\\
		& \geq \tilde{\gamma} \left[ \norm{v}^2_{L^2(I;Y)} +  \norm{w}^2_{L^2(I;Y)} \right] \label{lb_L}
		\end{align}
		where $\ds \tilde{\gamma} = \min \left \{ 1 - \wti{M}_1 \wti{M}_2 \norm{\bar p}_{\Liy} , \alpha - \wti{M}_1 \wti{M}_2 \norm{\bar p}_{\Liy} \right \}$. By possible further reduction of $\delta_3$ it can be guaranteed that $\tilde \gamma >0$, see \eqref{est:p}. Then by \eqref{ker_vw}, we obtain,
		\begin{align*}
		\calL''(\bar{y}, \bar{u}, y_0, \bar p, \bar p_1)((v,w),(v,w)) &\ge \frac{\tilde{\gamma}}{2} \left[ \norm{v}^2_{L^2(I;Y)} +  \norm{w}^2_{L^2(I;Y)} \right] + \frac{\tilde{\gamma}}{2 \wti{M}_2} \norm{v}^2_{\Wiay},\\
		&\ge \frac{\tilde{\gamma}}{2 \wti{M}_2} \norm{v}^2_{\Wiay} + \frac{\tilde{\gamma}}{2} \norm{w}^2_{L^2(I;Y)}.
		\end{align*}
		\nin By selecting $\ds \bar{\gamma} = \min \left \{ \frac{\tilde{\gamma}}{2 \wti{M}_2}, \frac{\tilde{\gamma}}{2} \right \}$, we obtain the positive definiteness of $\calL''$, i.e.
		\begin{equation}\label{sosc}
		\calL''(\bar{y}, \bar{u}, y_0, \bar p, \bar p_1)((v,w),(v,w)) \ge \bar{\gamma} \norm{(v,w)}^2_{\Wiay \times U}, \ y_0 \in B_Y(\delta_3), \ (v,w) \in \text{ker } \calE.
		\end{equation}			
	\end{proof}


	\section{Lipschitz stability of optimal controls}\label{sec-lip}
	
	\subsection{Generalized equations}
	We recall a result on parametric Lipschitz stability of  solutions of generalized equations  in a form due to Dontchev \cite{D:1995}. For this purpose we  consider
	\begin{equation}\label{eq_gen}
	0 \in \calF(x) + \calN(x),
	\end{equation}
	\nin where $\calF$ is a $C^1$-mapping between two Banach spaces $\calX$ and $\calZ$, and $\calN: \calX \mapsto 2^{\calZ}$ is a set-valued mapping with a closed graph. Let $\bar x$ be a solution of \eqref{eq_gen}. The generalized equation is said to be strongly regular at  $\bar x$, if there exist  open balls $B_{\calX}(\bar x, \delta_x)$ and $B_{\calZ}(0, \delta_z)$ such that for all $\bs{\beta} \in B_{\calZ}(0, \delta_z)$ the linearized and perturbed equation
	\begin{equation}\label{eq4.2}
	\bs{\beta} \in \calF(\bar x) + \calF'(\bar x)(x - \bar x) + \calN(x)
	\end{equation}
	\nin admits a unique solution $x = x(\bs{\beta})$ in $B_{\calX}(\bar x, \delta_x)$, and the mapping $\bs{\beta} \mapsto x$ is Lipschitz continuous  from  $B_{\calZ}(0, \delta_z)$ to $B_{\calX}(\bar x, \delta_x)$. We have the following result   which allows to conclude  local stability of the perturbed nonlinear problem from the stability of the linearized one.
	\begin{thm}\label{thm_gen_eq}
		Let $\bar x$ be a solution of \eqref{eq_gen} and assume that \eqref{eq_gen} is strongly regular $\bar x$. Then there exist open balls $B_{\calX}(\bar x, \delta'_x)$ and $B_{\calZ}(0, \delta'_z)$ such that for all $\bs{\beta} \in B_{\calZ}(0, \delta_z)$, the perturbed equation
		\begin{equation*}
		\bs{\beta} \in \calF(x) + \calN(x)
		\end{equation*}
		\nin has a unique solution $x = x(\bs{\beta})$ in $B_{\calX}(\bar x, \delta'_x)$, and the solution mapping $\bs{\beta} \mapsto x(\bs{\beta})$ is Lipschitz continuous from $B_{\calZ}(0, \delta'_z)$ to $B_{\calX}(\bar x, \delta'_x)$.
	\end{thm}
	
	\subsection{The perturbed optimal control problem}
	\nin To cast the first order optimality system \eqref{eq:rp_h1} as a special case of \eqref{eq_gen}, let $(\bar y, \bar u)$ be a local solution of \eqref{P}  with initial datum $y_0 \in B_V(\delta_3)$, and let $\bar p$ denote the associated adjoint state. Then optimality system for  \eqref{P} can be expressed as:
	\begin{equation}\label{eq_genF}
	0 \in \calF(\bar y, \bar u, \bar p) + \left( 0,0,0,\partial \mb{I}_{\Uad}(\bar u) \right)^{T},
	\end{equation}
	where the function $\calF:\calX \to \calZ $ with
	\begin{equation}
	\calX=\Wiay \times (U \cap C(\bar I;\calU)) \times \Wiay, \text{ and } \calZ= \Liy \times (U \cap C(\bar I;\calU)) \times \Liy \times V
	\end{equation}
	\nin is given by
	\begin{equation}
	\calF( y,  u,  p) = \bpm  y_t - A y - F(  y) - Bu \\ \alpha u - B^*p \\ - p_t - A^*  p - F'({y})^*  p + {y} \\  y(0) - y_0 \epm,
	\end{equation}
	and
	\begin{equation}
	\partial \mb{I}_{\Uad}( u) = \left\{ \tilde{u} \in U \cap C(\bar I;\calU): \ipp{\tilde{u}(t)}{v - u(t)}_{\calU} \leq 0,\, \forall\, t \in I, \ v \in B_{\eta}(0)  \right\},
	\end{equation}
	\nin with $\ds B_{\eta}(0) = \left\{ v \in \calU: \norm{v}_{\calU} \leq \eta \right\}$. In order to apply Theorem \ref{thm_gen_eq} to \eqref{eq_genF}, we need to show strong regularity of this equation at the reference solution $(\bar y, \bar u, \bar p)$ of \eqref{eq_genF}. First we note that $\calF$ is continuously differentiable  by \eqref{asp_3}. Observe also that for  $\bs{\beta} = (\beta_1, \beta_2, \beta_3, \beta_4) \in \calZ$
	the generalized equation
	\begin{equation}\label{inc_b_gen}
	\bs{\beta} \in \calF(y,u,p) + \left( 0,0,0, \partial \mb{I}_{\Uad}( u) \right).
	\end{equation}
	is the first order optimality system for
	\begin{subequations}\label{nl_ocp}
		\begin{multline}
		\inf_{\bem y \in \Wiay \\ u \in \Uad \eem} \hat{J}(y,u) = \inf_{\bem y \in \Wiay \\ u \in \Uad \eem} \ \frac{1}{2} \int_{0}^{\infty} \norm{y}_Y^2 dt \\+ \frac{\alpha}{2} \int_{0}^{\infty} \norm{u}_{\calU}^2 dt - \int_{0}^{\infty} \ipp{y}{\beta_3}_{Y} dt - \int_{0}^{\infty} \ipp{u}{\beta_2}_Y dt,
		\end{multline}
		\nin subject to
		\begin{empheq}[left=\empheqlbrace]{align}
		y_t &= A y + Bu + F(y) + \beta_1 & \text{ in } L^2(I;Y),\\
		y(0) &= y_0 + \beta_4 & \text{ in } V.
		\end{empheq}
	\end{subequations}
	
	The linearized version of \eqref{inc_b_gen}  is given by
	\begin{equation}\label{lin_b_gen}
	\bs{\beta} \in \calF(\bar y, \bar u, \bar p) + \calF'(\bar y, \bar u, \bar p)(y - \bar y,u - \bar u,p - \bar p) +  \left( 0,0,0, \partial \mb{I}_{\Uad}( u) \right),
	\end{equation}
	\nin or equivalently
	\begin{subequations}\label{eq:kk9}
		\begin{empheq}[left=\empheqlbrace]{align}
		y_t = A y + B u + F'(\bar{y})(y - \bar{y}) - F(\bar{y}) + \beta_1,\\
		\alpha u - B^*p + \partial \mb{I}_{\Uad}( u) \ni \beta_2, \label{u-pur}\\
		- p_t - A^* p - {F'(\bar{y})^*}\, p + y - [{F'(\bar{y})^*}\,{\bar{p}}]'(y - \bar{y}) = \beta_3, \label{p-pur}\\
		y(0) = {y}_0 + \beta_4.
		\end{empheq}
	\end{subequations}
	\nin This is the  optimality system of the following perturbed linear-quadratic optimization problem
	\begin{subequations}\label{lin_ocp}
		\begin{multline}
		\inf_{\bem y \in \Wiay \\ u \in \Uad \eem} \hat{J}(y,u) = \inf_{\bem y \in \Wiay \\ u \in \Uad \eem} \frac{1}{2} \int_{0}^{\infty} \norm{y}_Y^2 dt + \frac{\alpha}{2} \int_{0}^{\infty} \norm{u}_{\calU}^2 dt \\ - \frac{1}{2} \int_{0}^{\infty} \ipp{[F'(\bar{y})^*\bar{p}]'y - \bar y}{y- \bar y}_{Y} dt
		- \int_{0}^{\infty} \ipp{y}{\beta_3}_{Y} dt - \int_{0}^{\infty} \ipp{u}{\beta_2}_{\mathcal{U}}\, dt,
		\end{multline}
		\nin subject to
		\begin{empheq}[left=\empheqlbrace]{align}
		y_t &= A y + B u + F'(\bar{y})(y - \bar{y}) - F(\bar{y}) + \beta_1 & \text{ in } L^2(I;Y),\\
		y(0) &=  {y}_0 + \beta_4 & \text{ in } V.
		\end{empheq}
		
	\end{subequations}

	\begin{rmk}\label{rmk-pur}
		Concerning some basic properties of problem \eqref{lin_ocp} we can proceed as in Lemma 4.7, Remark 4.2, and \textit{Step} (ii) of Theorem 2.1 of \cite{KP:2022}. First,  note that the second order sufficient optimality condition  in the sense of  \eqref{snd_opt} for \eqref{P} and for \eqref{lin_b_gen} coincide. By Proposition \ref{le4.7} it is satisfied by each local solution to \eqref{P} if $y_0\in B_V(\delta_3)$.
		Then there exists a bounded neighborhood $\hat V$ of the origin in $\calZ$ such that for each $\bs{\beta} \in \hat V$ there exists a unique solution $ ( y_{(\bs{\beta})}, u_{(\bs{\beta})}, p_{(\bs{\beta})}) \in \Wiay \times U \times \Wiay$ to the perturbed linearized problem \eqref{lin_ocp} or equivalently of \eqref{eq:kk9}. Moreover $( y_{(\bs{\beta})}, u_{(\bs{\beta})}) \in \Wiay \times U $ depends Lipschitz continuously on $\bs{\beta}\in \hat V$.
		For the latter we can proceed as in Step (iii) of the proof in \cite[Theorem 2.1]{KP:2022}. Note, however, that at this point we have not yet guaranteed that $\bs{\beta} \to u_{(\bs{\beta})}$ is Lipschitz continuous with values in $C(\bar{I};\mathcal{U})$, which is necessary due to the norm on $\calX$. This, and the Lipschitz continuity of $\bs{\beta} \to u_{\bs({\beta})}$ will be established in the following subsection.
		In the proof of Lemma \ref{lem:beta_p} we shall also require that $\ds \|{\beta_2}\|_{C(\bar I;\mathcal{U})} \le \frac{\alpha \eta}{2}$, which is henceforth assumed to hold.
		%
	\end{rmk}
	
	\subsection{Lipschitz stability of optimal controls, states and adjoint states}
	
	\nin Now we will show Lipschitz stability of optimal control, state and adjoint state,     in a neighborhood of a local solution $(\bar y, \bar u)$ of \eqref{P}  with initial datum $y_0 \in B_V(\delta_3)$, and associated adjoint state  $\bar p$.  This will be achieved once
	Lipschitz continuity of the solution mapping $\bs{\beta} \in \hat V \mapsto ( y_{(\bs{\beta})}, u_{(\bs{\beta})}, p_{(\bs{\beta})})$ of the perturbed linearized problem \eqref{lin_ocp} is proven. 
	For this purpose we require the following result for the adjoint states.

	\begin{lemma}\label{lem:beta_p}
		Let \eqref{asp_1}-\eqref{asp_4}   hold and let $(\bar y, \bar u)$, and $\bar p$ be a local solution and associated adjoint state to \eqref{P} corresponding to an initial datum $y_0 \in B_V(\delta_3)$. Then the mapping $\bs \beta \mapsto p_{(\bs \beta)}$ is continuous from $\hat V \subset \mathcal{Z}$ to $\Wiay$.
	\end{lemma}
	
	\begin{proof} The proof related to that of Proposition 4.8 in \cite{KP:2023}, but it is sufficiently different so that we prefer to provide it here.\\

		\nin \uline{\textit{Step 1:}} For $\bs{\beta} \in \hat V$, with $\hat V$ as in Remark \ref{rmk-pur}, there exists a unique solution $\ds (y_{(\bs \beta)}, u_{(\bs \beta)}, p_{(\bs \beta)})$ to the perturbed linear system \eqref{lin_ocp}. The perturbed costate equation, and the constraint on the control can be expressed as
		\begin{subequations}
			\begin{align}
			- \partial_t p_{\ba} - A^* p\ba - F'(\bar{y})^*p\ba + y\ba - [F'(\bar{y})^*\bar{p}]'(y\ba - \bar{y}) &= \beta_1 \quad \text{ in } \Liy, \label{adjTP_1}\\
			\ip{\alpha u\ba - B^* p\ba - \beta_2}{w - u\ba}_U &\geq 0 \quad \text{for all } w \in \Uad. \label{adjTP_2}
			\end{align}
		\end{subequations}
		\nin Since $p\ba \in \Wiay \subset C(\bar I; Y)$ and $\beta_2\in C(\bar I; Y)$, this implies that $u\ba \in C(\bar I; \calU)$.\\
		
		\nin \uline{\textit{Step 2:} (Boundedness of $\{ p_{(\bs{\beta})}:  \bs{\beta} \in \hat{V} \}$ in $\Wiay$.) }
		Since $\hat{V}$ is assumed to be bounded, by  Remark \ref{rmk-pur} there exists a constant $M$ such that
		\begin{equation*}
		\norm{y_{(\bs{\beta})}}_{\Wiay} + \norm{u_{(\bs{\beta})}} \leq M \quad \text{for all } \bs{\beta} \in \hat{V}.
		\end{equation*}

		\nin To argue the boundedness of $p_{(\bs{\beta})}$, replacing $\bar y$ by $y\ba -[F'(\bar y)^*\bar p]'(y\ba -\bar y)$ we can first proceed as in \textit{Step 3} of proof of Proposition \ref{prop:adj} to obtain the boundedness of $\{\|p_{(\bs{\beta})}\|_{\Wiya}: \bf{\beta}\in \hat V\}$. Subsequently we proceed as in  \textit{Steps 1} and \textit{3} of Proposition \ref{prop-p-W} to obtain the boundedness of  $\{\|p_{(\bs{\beta})}\|_{\Wiay}: \bf{\beta}\in \hat V\}$.\\
		
		\nin \uline{Step 3: (Continuity of $p_{(\bs \beta)}$ in $\Wiay$).} Let $\{ \bs \beta_n \}$ be a convergent sequence in $\hat V$ with limit $\bs \beta$. Since $\ds \left\{ \norm{p_{({\bs \beta}_n)}}_{\Wiay}: n \in \BN \right\}$ is bounded, there exists a subsequence $\{ \bs \beta_{n_k} \}$ such that $\ds p_{(\bs \beta_{n_k})} \rightharpoonup \tilde{p}$ weakly in $\Wiay$ and strongly $L^2(0,T;V)$ for every $T\in (0,\infty)$, see e.g. \cite[Satz 8.1.12, p 213]{EE:2004}. Passing to the limit in the variational form of
		\begin{equation*}
		- \partial_t p_{(\bs \beta_{n_k})}  - A^* p_{(\bs \beta_{n_k})} - F'(\bar{y})^*p_{(\bs \beta_{n_k})} + y - [F'(\bar{y})^*\bar{p}]' \left(y_{(\bs \beta_{n_k})} - \bar{y} \right) = (\bs \beta_{n_k})_1,
		\end{equation*}
		\nin we obtain
		\begin{equation}
		- \partial_t \tilde{p} - A^* \tilde{p} - F'(\bar{y})^*\tilde{p} + y_{(\bs \beta)} - [F'(\bar{y})^*\bar{p}]'\left( y_{(\bs \beta)} - \bar{y} \right) = (\bs \beta)_1.
		\end{equation}
		\nin Since the solution to this equation is unique we have $\ds p_{(\bs \beta_n)} \rightharpoonup p_{(\bs \beta)}$ weakly in $\Wiay$. To obtain strong convergence we set $\delta \bs \beta = \bs \beta_n - \bs \beta, \ \delta p = p_{(\bs \beta_n)} - p_{(\bs \beta)}, \ \delta y = y_{(\bs \beta_n)} - y_{(\bs \beta)}$. From \eqref{adjTP_1} we derive that
		\begin{equation}\label{eq:del_p}
		- \partial_t (\delta p) - A^* (\delta p) - F'(\bar{y})^*(\delta p) + (I - [F'(\bar{y})^*\bar{p})]'(\delta y) = (\delta \bs \beta)_1,
		\end{equation}
		holds in $L^2(I;Y)$.  We shall employ a duality argument to obtain a bound on $\delta p$. Moreover we shall argue that the constraint$ \|u\ba(t)\|_{\calU}\le \eta$ is inactive for all $t$ sufficiently large. Indeed, since  $\ds p_{(\bs \beta)} \in W_\infty(\mathcal{D}(A),Y)$ we have  $\ds \lim_{t \rightarrow \infty} p_{(\bs \beta)}(t) = 0$ in $V$. Hence there exists $\hat T$ such that $\ds \frac{1}{\alpha} \norm{B^* p_{(\bs \beta)} (t) }_{\calU} \leq \frac{\eta}{4}$ for all $t \geq \hat T$, and by the choice of $\hat V$, where we assumed that $\|\beta_2\|_{C(\bar I;\mathcal{U})} \le \frac{\alpha \eta}{2}$, see Remark \ref{rmk-pur}, we also have that
		\begin{equation}\label{eq:kk10}
		\norm{u_{(\bs{\beta})}(t)}_{\calU} = \norm{ \mathbb{P}_{\mathcal{U}_{ad}} \left[ \frac{1}{\alpha} \left(B^* p_{(\bs{\beta})}(t) + \beta_2(t)\right) \right]}_{\calU} = \frac{1}{\alpha} \norm{B^* p_{(\bs{\beta})}(t)	+ \beta_2(t)}_{\calU} \leq \frac{3 \eta}{4},
		\end{equation}
		\nin i.e. the constraint is inactive for $t \geq \hat T$.
		
		Henceforth let $r,z,w=-Kz$ and $\tilde \delta$ be as introduced in \eqref{eq:zeq_new} and recall that $w\in U_{ad}$.  Then we obtain
		\begin{align*}
		\langle \delta p,  B( Kz - u_{( \bs \beta_{n_k})})&\rangle_{L^2(I;Y)} \leq \int_{0}^{\hat T} \ip{p_{(\bs \beta_{n_k})}(t) - p_{(\bs \beta)}(t)}{B( Kz(t) - u_{( \bs \beta_{n_k} )}(t) )}_Y dt \\
		& \hspace{1.3cm} + \int_{\hat T}^{\infty} \ip{B^*(p_{(\bs \beta_{n_k})}(t) - p_{(\bs \beta)}(t))}{Kz(t)- u_{( \bs \beta_{n_k} )}(t) }_{\mathcal{U}}  \,  dt,\\
		&\leq \int_{0}^{\hat T} \norm{B^*(p_{(\bs \beta_{n_k})}(t) - p_{(\bs \beta)}(t))}_Y\norm{Kz(t) - u_{( \bs \beta_{n_k} )}(t) }_{\mathcal{U}}dt\\
		& + \int_{\hat T}^{\infty} \ip{(\alpha u_{(\bs \beta_{n_k})}(t) - \beta_{n_k,2}(t)) - (\alpha u_{(\bs \beta)}(t) - \beta_2(t))}{ Kz(t)  - u_{( \bs \beta_{n_k}) }(t)}_{\mathcal{U}}\,dt,
		\end{align*}
		where we used that $w\in U_{ad}$ and feasibility  of $u_{(\bs \beta)}(t)$ for $t\ge \hat T$. Consequently we have
		\begin{equation}\label{eq:kk11}
		\begin{array}l
		\langle \delta p,  B( Kz - u_{( \bs \beta_{n_k})})\rangle_{L^2(I;Y)} \leq
		\left( \norm{B}_{\calL(\calU,Y)} \norm{p_{(\bs \beta_{n_k})} - p_{(\bs \beta)}}_{L^2(0,\hat T; Y)} + \alpha \norm{u_{(\bs \beta_{n_k})} - u_{(\bs \beta)}}_U \right.\\
		\hspace{1.3cm} \left. + \norm{\beta_{n_k,2} - \beta_2}_U \right) \left( \norm{K}_{\calL(Y,\calU)} \norm{z}_{W_{\infty}} + \norm{u_{( \bs \beta_{n_k} )}}_U \right).
		\end{array}
		\end{equation}
		\nin Let $R_{\tilde \delta} = \left\{ r \in L^2(I;Y): \norm{r}_{L^2(I;Y)} \leq \delta \right\}$. For arbitrary $r \in R_{\tilde \delta}$ let $z = z(r)$
		denote the solution to \eqref{eq:zeq_new} with $\bs \beta \in \hat V$. We find
		\begin{equation*}
		\ip{\delta p}{r}_{L^2(I;Y)} = \ip{(I - [F'(\bar{y})^*\bar{p}]')(\delta y) - \delta \beta_1}{z}_{L^2(I;Y)} + \ip{\delta p}{B(Kz_{(\bs \beta)} - u_{(\bs \beta)})}_{L^2(0,\hat T; Y)},
		\end{equation*}
		\nin and thus, using \eqref{eq:kk11} and \eqref{eq:z_est_new}, we obtain for some $C_2 > 0$,
		\begin{align}
		\norm{\delta p}_{L^2(I;Y)} &= \sup_{r \in R_{\tilde \delta}} \ \ip{\delta p}{r}_{L^2(I;Y)} \nonumber \\
		&\leq C_2 \left( \norm{\delta y}_{_{\Wiay}} + \|\ipp{\delta \beta_1}{\delta \beta_2}\|_{L^2(I;Y) \times U} + \norm{\delta p}_{L^2(0,\hat T;Y)} + \norm{\delta u}_U \right). \label{est_p_b}
		\end{align}
		\nin Since $\norm{\delta y}_{_{\Wiay}} \to 0,\ \norm{\delta p}_{L^2(0,\hat T;Y)} \to 0,\ \|{\ipp{\delta \beta_1}{\delta \beta_2}\|_{L^2(I;Y) \times U}} \to 0$ for $n \to 0$ this implies that $\norm{\delta p}_{L^2(I;Y)} \to 0$.  To obtain convergence to 0 of $\delta p$ in $W_\infty(V,V')$ we can proceed similarly as in Step 1 of the proof of Proposition \ref{prop-p-W}. For a moment we now emphasize the dependence of $\delta p$ on $n$ and write  $\delta p_n$ instead.
		Since $\delta p_n \in L^2(I;Y)$ of each $n$  there exists a  monotonically increasing sequence $\ds \{ t(n)_k \}_{n \in \BN}$ with $\lim_{k\to\infty} t(n)_k=\infty$ and  $\lim_{k\to\infty}\delta p_n(t(n)_k) = 0 $ in $Y$. Now, replacing $-\bar y$ by$ -(I -[F'(\bar y)^*\bar p]'\delta y_n + (\delta \beta_n)_1$ we obtain analogously to \eqref{eq:kk12}
		\begin{equation*}
		{\theta} \int_{0}^{\infty} \norm{\delta p_n}^2_V\leq \left(\frac{c_1^2}{\theta} + 1 + 2\rho \right) \norm{\delta p_n}^2_{L^2(I;Y)} + 2 \norm{ -(I -[F'(\bar y)^*\bar p]'\delta y_n + (\delta \beta_n)_1}^2_{L^2(I;Y)} .
		\end{equation*}
		The right hand side tends to $0$ for $n\to \infty$, and hence $\lim_{n\to\infty} \|\delta p_n\|_{L^2(I;V)}=0$. Utilizing this fact in \eqref{eq:del_p} the convergence  $\lim_{n\to\infty} \|\delta p_n\|_{W_\infty(V,V')}=0$ follows. Since $\delta p_n \in L^2(I;V)$ we also have
		$\lim_{k\to\infty}\delta p_n(t(n)_k) = 0 $ in $V$, for some
		monotonically increasing sequence $\ds \{ t(n)_k \}_{n \in \BN}$ with $\lim_{k\to\infty} t(n)_k=\infty$, for each $n$.\\

		\nin By taking the inner product in \eqref{eq:del_p} with $-A^*_{\rho} = -A^* + \rho I$ and continuing as in the \textit{Step 3} of Proposition \ref{prop-p-W}, we obtain,
		\begin{equation*}
		\int^{\infty}_0\! \| A_\rho^*(\delta p)(t)\|^2_Y\,dt  \le  C_3 \!\int_0^{\infty} \left( \|F'(\bar y)^*(\delta p)(t)\|^2_Y + \norm{[F'(\bar{y})^*\bar{p}]'(\delta y)(t)}^2_Y + \rho \,\| (\delta p)(t)\|^2_Y + \|(\delta \bs{\beta})_1(t)\|^2_Y \right)\,dt,
		\end{equation*}
		\nin for a constant $C_3$ independent of $n$.  Since $\norm{\delta y}_{\Wiay} \to 0,\ \norm{\delta p}_{L^2(I;V)} \to 0,\ \norm{(\delta \bs \beta)_1}_{\Liy} \to 0$ for $n \to 0$ this implies that $\norm{\delta p}_{\LiA} \to 0$. Utilizing this fact in  \eqref{eq:del_p} we find that $\lim_{n\to \infty}\ds \norm{\delta p}_{\Wiay}    = 0$.
		
	\end{proof}

	\begin{prop}\label{prop:est_p}
		Let \eqref{asp_1}-\eqref{asp_4}  hold and let $(\bar{y}, \bar{u})$, and $\bar p$ denote a local solution and associated adjoint state to \eqref{P} corresponding to an initial condition $y_0 \in B_V(\delta_3)$. Then there exists $\varepsilon > 0$ and $C > 0$ such that  for all  $\bs{\hat{\beta}}$ and  $\bs{\beta} \in \hat{V} \cap B_{\calZ}(\varepsilon)$
		\begin{multline}\label{est:p_dif}
		\norm{\hat{y}_{(\hat{\bs \beta})} - y_{(\bs \beta)}}_{\Wiay} + \norm{\hat{p}_{(\hat{\bs \beta})} - p_{(\bs \beta)}}_{\Wiay} + \norm{u_{(\hat{\bs \beta})} - u_{(\bs \beta)}}_{C(\bar I;\calU)} \leq C  \norm{\hat{\bs \beta} - \bs \beta}_{\calZ}
		\end{multline}
		holds.
	\end{prop}
	\begin{proof}
		The  Lipschitz continuity of $( y_{(\bs{\beta})}, u_{(\bs{\beta})}) \in \Wiay \times U $ for $\bs{\beta}$  in a neighborhood of the origin was already addressed in Remark \ref{rmk-pur}. We need to assert the extra Lipschitz continuity of $u_{\bs{\beta}} \in C(\bar{I};\mathcal{U})$ and the Lipschitz continuity of $p_{\bs{\beta}}$.

		Let us henceforth set $(y,u,p) = \left( y_{(\bs{\beta})}, u_{(\bs{\beta})}, p_{(\bs{\beta})} \right)$, and $\left( \hat{y}, \hat{u}, \hat{p} \right) = \left( \hat{y}_{(\bs{\hat{\beta}})}, \hat{u}_{(\bs{\hat{\beta}})}, \hat{p}_{(\bs{\hat{\beta}})} \right)$. We also set $\delta \bs{\beta} = \bs{\beta} - \hat {\bs{\beta}}, \ \delta p = p_{(\bs \beta)} - \hat p_{(\hat {\bs \beta})}, \ \delta y = y_{(\bs \beta)} - \hat y_{(\hat{\bs \beta})}$.
		\nin Then we obtain the equation
		\begin{equation}\label{eq:dp}
		- \partial_t (\delta p) - A^* (\delta p) - F'(\bar{y})^*(\delta p) + (I - [F'(\bar{y})^*\bar{p})]'(\delta y) = (\delta \beta_1) \in \Liy \subset \LiVp.
		\end{equation}
		\nin Since $\ds \norm{[F'(\bar{y})^*\bar{p})]'(\delta y)}_{\Liy} \lesssim \norm{\delta y}_{\Wiay}$   equation  \eqref{eq:dp} is well-defined on $\LiVp$.  Hence we can at first  apply the same technique as in \cite[Proposition 2, pg. 32]{KP:2022} to obtain the Lipschitz continuous dependence of $\delta p$ with respect to $\bs \beta$. Indeed since $\bar p \in W_\infty(V,V')$ there exists $T$ such that $\frac{1}{\alpha}\|B^*\bar p(t)\|_Y \le \frac{\eta}{2}$, for all $t\ge T$, and thus the control constraint is inactive on $[T,\infty$. Utilizing the continuity established in Lemma \ref{lem:beta_p} there exists $\epsilon > 0$ such that $\ds \frac{1}{\alpha}\|B^* p_{(\bs \beta)}(t) + \beta_2(t)\|_Y \le \frac{3\eta}{4}$, for all $t\ge T$ and all  $\bs{\beta} \in \hat{V} \cap B_{\calZ}(\varepsilon)$, and thus by \eqref{adjTP_2} the control $u_{\bs \beta}$ is inactive for these values of $\bs \beta$ and $t$.  We can now proceed as in the mentioned reference to assert that there exists a constant $C_1$ independent of $\bs{\beta} \in \hat{V} \cap B_{\calZ}(\varepsilon)$ such that
		\begin{equation}
		\norm{\delta p}_{\Wiv} \leq C_1 \left( \norm{\delta y}_{\Wiay} + \norm{\delta u}_{U} + \norm{\delta \bs{\beta}}_{\calZ} \right).
		\end{equation}
		Next we need to obtain the Lipschitz continuity of $p$ in $\Wiay$.
		For this purpose  we take the inner product in \eqref{eq:dp} with $-A^*_{\rho}p = (-A^* + \rho I)p$ and continue as in the \textit{Step 3} of Proposition \ref{prop-p-W}. We find
		\begin{equation}
		\norm{\delta p}_{\LiA} \leq C_2 \left( \norm{\delta p}_{\Liy} + \norm{\delta y}_{\Wiay} + \norm{\delta \beta_1}_{\Liy} \right).
		\end{equation}
		\nin Combining the two inequalities from above, we obtain
		\begin{equation}
		\norm{\delta p}_{\LiA} \leq C_3 \left( \norm{\delta y}_{\Wiay} + \norm{\delta u}_{U} + \norm{\delta \bs{\beta}}_{\calZ} \right).
		\end{equation}
		Then by \eqref{eq:dp}, we obtain the Lipschitz continuity of $p_t$ in $L^2(I;Y)$ and of $\delta y$. Combining these results we deduce that
		\begin{equation}\label{p-lip-s}
		\norm{\delta p}_{\Wiay} \leq C_3 \left( \norm{\delta y}_{\Wiay} + \norm{\delta u}_{U} + \norm{\delta \bs{\beta}}_{\calZ} \right).
		\end{equation}
		\nin We also have
		\begin{equation*}
		u_{(\bs \beta)} = \mathbb{P}_{\mathcal{U}_{ad}} \left[ -\frac{1}{\alpha}\left(B^* p_{(\bs \beta)} + \beta_2 \right)\right] \in U \cap C(\bar I;\calU),
		\end{equation*}
		\nin and thus
		\begin{align*}
		\norm{\delta u(t)}_{\calU} &\leq \norm{\mathbb{P}_{\mathcal{U}_{ad}} \left[ -\frac{1}{\alpha} \left(B^*\hat{p}_{(\hat{\bs{\beta}})}(t) + \hat{\beta}_2(t) \right)\right] - \mathbb{P}_{\mathcal{U}_{ad}} \left[ -\frac{1}{\alpha}\left(B^* p(t)_{(\bs \beta)} + \beta_2(t) \right)\right]}_{\calU} \\
		&\leq \frac{1}{\alpha} \left( \norm{B^*} \norm{\delta p(t)}_{Y} + \norm{\delta\beta(t)}_{\calU} \right).
		\end{align*}
		This yields
		\begin{equation}\label{u-lip}
		\norm{\delta u}_{C(\bar I;\calU)} \leq C_4 \left( \norm{\delta p}_{\Wiay} + \norm{\delta \beta_2}_{C(\bar I;\calU)} \right),
		\end{equation}
		\nin Combining \eqref{p-lip-s} and \eqref{u-lip}, there exists a constant $L$ such that
		\begin{equation}
		\norm{\hat{y}_{(\hat{\bs \beta})} - y_{(\bs \beta)}}_{\Wiay} + \norm{\hat{p}_{(\hat{\bs \beta})} - p_{(\bs \beta)}}_{\Wiay} + \norm{\hat{u}_{(\hat{\bs \beta})} - u_{(\bs \beta)}}_{U \cap C(\bar I;Y)} \leq L \norm{\hat{\bs \beta} - \bs \beta}_{\calZ}
		\end{equation}
		for all $\ds \bs{\hat{\beta}}$ and $\ds \bs{\beta} \in \hat{V} \cap B_{\calZ}(\varepsilon)$.
	\end{proof}	
	
	\nin We have  concluded the verification of the strong regularity condition for problem \eqref{eq_genF} and can conclude the following result from Theorem  \ref{thm_gen_eq}.
	
	\begin{clr}\label{clr:lip_con}
		Let the assumptions \eqref{asp_1}-\eqref{asp_4} hold and let $(\bar{y}, \bar{u})$ be a local solution of \eqref{P} corresponding to an initial datum $y_0 \in B_V(\delta_3)$. Then there exist $\delta_4 > 0$, a neighborhood $\hat U = \hat U(\bar{y}, \bar{u}, \bar p) \subset \Wiay \times (U \cap C(\bar{I};\calU)) \times \Wiay$, and a constant $\mu > 0$ such that for each $\tilde{y}_0 \in B_V(y_0; \delta_4)$ there exists a unique  $(y(\tilde{y}_0), u(\tilde{y}_0), p(\tilde{y}_0)) \in \hat U(\bar{y}, \bar{u}, \bar p)$ satisfying the first order condition, and
		\begin{multline}\label{eq:clr_lc_2}
		\norm{ \left( y(\hat{y}_0), u(\hat{y}_0), p(\hat{y}_0) \right) \!-\! \left(y(\check{y}_0), u(\check{y}_0), p(\check{y}_0) \right)}_{\Wiay \times (U \cap C(\bar{I};\calU)) \times \Wiay} \\\leq   \mu \norm{\hat{y}_0 - \check{y}_0}_{V},
		\end{multline}
		for all $\hat{y}_0, \check{y}_0 \in B_V(y_0, \delta_4)$. Moreover $(y(\tilde{y}_0), u(\tilde{y}_0))$ is a local solution of \eqref{P}.
	\end{clr}
	
	\nin Above $B_V(y_0, \delta_4)$ denotes the ball of radius $\delta_4$ and center $y_0$ in $V$.
	
	\section{Differentiability of the cost-functional and  HJB equation}\label{sec-HJB}

	As a consequence of the Corollary \ref{clr:lip_con} and Proposition \ref{le4.7} concerning the second order sufficient optimality condition,  for each  $y_0 \in B_V(\delta_3)$ there exists a neighborhood within which
	the local minimal value function $\mathcal{V}$ is well-defined and the corresponding controls, states and adjoint states depend Lipschitz continuously on the initial data. Consequently the local  value function itself is locally Lipschitz continuous. Exploiting the structure of the cost functional in \eqref{P} Fr\'{e}chet differentiability of the minimal value function can be obtained. We continue to use the notation for $B_V(y_0, \delta_4)$ of Corollary \ref{clr:lip_con} above and locally optimal solutions are understood in the sense of $\hat U$.
	
	
	\begin{thm}\label{thm-CD-r}
		(Sensitivity of Cost) Let the assumptions \eqref{asp_1}-\eqref{asp_4} hold and let $(\bar{y}, \bar{u})$ be a local solution of \eqref{P} corresponding to an initial datum $y_0 \in B_V(\delta_3)$.
		Then  for each $\hat{y}_0 \in B_V(y_0, \delta_4)$ the local minimal value function associated to \eqref{P} is Fr\'{e}chet differentiable  from $B_V(y_0,\delta_4)$ to $\mathbb{R}$   with derivative given by
		\begin{equation} \label{eq:5.1}
		\V'(\hat y_0) = - p(0; \hat y_0),
		\end{equation}
		and the Riesz representor of $\V'$ lies in $C(U(y_0), V)$.
	\end{thm}
	
	\nin With Corollary \ref{clr:lip_con} available \eqref{eq:5.1} can be verified with the same techniques as the analogous one with $V$ replaced by $L^2(\Omega)$ given in \cite[Theorem 4.10]{KP:2022}. The claim concerning the Riesz representor follows from  \eqref{eq:5.1} and  Corollary \ref{clr:lip_con}.\\
	
	\nin For the final theorem we need an additional assumption.
	\begin{assumption}\label{asp_5}
		With the notation of the previous theorem there exists $T_0 \in (0,\infty)$ such that  $F ( y(\hat y_0,u)) \in C([0,T_{y_0}); V')$
		for all $\hat y_0 \in B(y_0,\delta_4)$ and all $u \in C([0,T_0); \mathcal{U})$,
		where $y(\hat y_0,u)$  denotes the solution to \eqref{eq_y} on $[0,T_0)$ with initial condition $\hat y_0$ and control $u$.
	\end{assumption}

	\begin{thm}\label{thm-HJB-r}
		Let   assumptions \eqref{asp_1}-\eqref{asp_5} hold and let $(\bar y, \bar u)$ be a local solution  of \eqref{P} corresponding to an initial datum $y_0 \in B_V(\delta_3)$.
		Then the following Hamilton-Jacobi-Bellman equation holds for the local minimal value function (in the sense of $\hat U$ from Corollary \ref{clr:lip_con}) on $B(y_0,\delta_4)$:
		\begin{equation}\label{hjb}
		\V'(y)(A y + F(y)) + \frac{1}{2} \norm{y}^2_Y + \frac{\alpha}{2} \norm{ \mathbb{P}_{\mathcal{U}_{ad}} \left(-\frac{1}{\alpha}B^*\V'(y) \right)}^2_Y + \left( B^* \V'(y),\mathbb{P}_{\mathcal{U}_{ad}} \left(-\frac{1}{\alpha}B^*\V'(y) \right) \right)_Y = 0,
		\end{equation}
		and the feedback law is given by
		\begin{equation}\label{eq:5.3}
		u =  \mathbb{P}_{\mathcal{U}_{ad}} \left(-\frac{1}{\alpha}  B^*\V'({y}) \right).
		\end{equation}
	\end{thm}
	\begin{proof}
		The structure of the proof is rather standard,  \cite[Theorem 5.1]{KP:2022}. But due to regularity issues  special treatment is required. Also compared to \cite{KP:2022}  we modify  some arguments to allow local rather than global solutions.
		Let  $\hat y_0 \in B_V(y_0,\delta_4)$  with associated local solution  and  adjoint state $(\hat{y}, \hat{u},\hat p) \in \hat U$ . In particular we have that $\ds \hat{u}(t) = \mathbb{P}_{\mathcal{U}_{ad}} \left( \frac{1}{\alpha}B^*{\hat p}(t) \right)$, and since $ \hat p \in C([0,\infty);V)$ it holds $ \hat u\in C([0,\infty);\mathcal{U})$. Let $\hat u_0$ denote the limit of $\hat{u}$ as time $t$ tends to $0$.
		Since $\hat{y} \in C([0,\infty); V)$ and since $B_V(y_0,\delta_4) $ is open there exists $\tau_{\hat y_0} \in (0,T_0]$  such that $\hat{y}(t) \in B_V(y_0,\delta_4) $, for all $t \in [0,\tau_{\hat y_0})$, and $\mathcal{V}$ is well-defined there.\\
		
		\nin {\em{Step 1.}}	Let us first prove that
		\begin{equation}\label{hjb.1}
		\V'(\hat y_0) \big(A \hat y_0 + F(\hat y_0) + B \hat u_0 \big) + \ell(\hat y_0, \hat u_0) = 0,
		\end{equation}
		where $\ds \ell(y,u)= \frac{1}{2}\|y\|_Y^2 + \frac{\alpha}{2}\|u\|_\mathcal{U}^2$. Since  by the previous theorem the Riesz representor of $\V'(\hat y_0)$ is an element of $V$, and since the arguments of $\V'(\hat y_0)$ are all contained in $V'$, the left hand side of the above equality is well-defined. Here we note that \eqref{asp_5} implies that $F(\hat y_0)\in V'$.\\
		
		\nin To verify the equality   we use the dynamic programming principle in the form
		\begin{equation}\label{v-lim}
		\frac{1}{\tau} \int_{0}^{\tau} \ell(\hat{y}(s), \hat{u}(s))ds + \frac{1}{\tau} \big( \V(\hat{y}(\tau)) - \V(\hat y_0) \big) = 0,
		\end{equation}
		for  $\tau \in (0,\tau_{\hat y_0}))$. By continuity of $\hat{y}$ and $\hat{u}$  in $Y$, respectively $\mathcal{U}$ at time $0$, the first term  converges to $\ds \ell(\hat y_0, \hat u_0)$ as $\tau \to 0$. To take $\tau \to 0$ in the second term we first consider
		\begin{equation}\label{eq:aux5}
		\frac{1}{\tau}\big( \hat{y}(\tau) -  \hat y_0 \big) = \frac{1}{\tau}\big(e^{A_{ext} \tau} \hat y_0 - \hat  y_0 \big) + \frac{1}{\tau} \int_{0}^{\tau} e^{A_{ext}(\tau - s)} \big[ F(\hat {y}(s)) + B \hat{u}(s) \big] ds,
		\end{equation}
		where $e^{A_{ext} t}$ denotes the extension of the semigroup $e^{At}$ on $Y$, to  $V'$. This follows by \cite[Theorem 5.5]{EN:2000} and an application of interpolation theory. Using \eqref{asp_5} and $ B\hat u \in C([0,\infty);Y)$ we can pass to the limit in $V'$ in \eqref{eq:aux5} to obtain that
		\begin{equation}\label{y-lim}
		\lim_{\tau \to 0^+}\frac{1}{\tau}\big( \hat{y}(\tau) - \hat  y_0 \big)  = A \hat  y_0 + F(\hat  y_0) + B \hat u_0 \text{ in } V'.
		\end{equation}
		Now we return to the second term in \eqref{v-lim} which we express as
		\begin{equation}\label{eq:aux6}
		\begin{aligned}
		\frac{1}{\tau} \big( \V(\hat  {y}(\tau)) - \V(\hat y_0) \big) =&\int_{0}^{1} \V' \big(\hat y_0 + s ( \hat  {y}(\tau) - \hat  y_0 ) \big)\, \frac{1}{\tau}( \hat {y}(\tau) - \hat  y_0 )\,ds \\
		=&\int_{0}^{1}\langle \hat p(0;(y_0 + s ( \hat {y}(\tau) - \hat  y_0 ))\; , \frac{1}{\tau}( \hat {y}(\tau) -\hat   y_0 ) \rangle_{V,V'}\, ds.
		\end{aligned}
		\end{equation}
		Using \eqref{y-lim} and since $y \to \hat p( 0;y)$ is continuous from $V$ to itself at  $\hat  y_0$  by Proposition \ref{clr:lip_con},  we can pass to the limit in \eqref{eq:aux6} to obtain
		\begin{equation}
		\lim_{\tau \to 0^+}\frac{1}{\tau} \big( \V(\hat{y}(\tau)) - \V(\hat  y_0) \big)= \V'(\hat  y_0) \big(A \hat  y_0 + F(\hat  y_0) + B \hat u_0 \big).
		\end{equation}
		\nin Now we can pass to the limit in \eqref{v-lim} and obtain \eqref{hjb.1}.\\
		
		\nin {\em{Step 2:}} For $u \in  \mathcal{U}_{ad}$  we define $\tilde{u}\in U_{ad}$ by,
		\begin{equation*}
		\wti{u}(t) =
		\begin{cases}
		u \quad \text{for } \tau \in (0,\hat \tau) \\
		\hat u(t) \quad \text{ for } \tau \in [\hat \tau, \infty)
		\end{cases}
		\end{equation*}
		\nin and set $\tilde{y} = y(\hat  y_0,\tilde{u})$ as the solution to  \eqref{eq_y}.
		Here $0< \hat \tau \le \min(T_{y_0})$ is chosen sufficiently small so that $\tilde u$ lies in the region of locality of  the local solution $\hat u$.
		Also note that  $\tilde y(t) \in B_V(\hat y_0,\delta_4)$, for all $t$ sufficiently small, and hence   $\V(\tilde{y}(t))$ is well-defined for all small $t$. By local optimality of $\hat u$ we have
		\begin{equation*}
		\frac{1}{\tau} \int_{0}^{\tau} \ell(\tilde{y}(s), \tilde u(s))ds + \frac{1}{\tau} \big( \V(\tilde{y}(\tau)) - \V(\hat  y_0) \big) \geq 0,
		\end{equation*}
		for all $\tau$ sufficiently small.
		We next  pass to the limit $\tau \to 0^+$  in the above inequality. This is trivial for the first term which tends to $\ell(\hat  y_0, u)$. For the second one we can argue as above, replacing $(\hat y, \hat u)$ by $(\tilde y, \tilde u)$ in \eqref{eq:aux5} and using \eqref{asp_5}, resulting in
		
		\begin{equation}\label{hjb.2}
		\V'(\hat  y_0) \big(A \hat y_0 + \calF(\hat y_0) + B u \big) + \ell(\hat y_0, u) \ge 0,
		\end{equation}
		for arbitrary $u \in \mathcal{U}_{ad}$.  This inequality is an equality if $u = \hat u_0$, and thus the quadratic function on the left had side of \eqref{hjb.2} reaches its minimum $0$ at $u =  \hat u_0$. This implies that
		$\ds \hat u_0 =  \mathbb{P}_{\mathcal{U}_{ad}} \left(-\frac{1}{\alpha} B^* \V'(\hat y_0) \right).$
		Inserting this expression into  \eqref{hjb.1} we obtain \eqref{hjb} at $y=\hat y_0$, and \eqref{eq:5.3} at $(y,u)=(\hat y_0, \hat u_0)$ holds as well. Since $\hat y_0 \in B(y_0,\delta_4)$ was chosen arbitrarily the proof is complete.
	\end{proof}
	
	\section{Applications}
	In the final section we discuss the applicability  of the presented theory for  selected examples.
	Throughout  $\Omega$ denotes an open connected bounded subset of $\BR^d,$  with $\Omega$ convex or with  a $C^{1,1}$ boundary $\Gamma$.  The associated space-time cylinder is denoted by $Q = \Omega \times (0,\infty)$ and the associated lateral boundary by  $\Sigma = \Gamma \times (0,\infty)$. By $\nu$ and $\partial_{\nu}$, we denote the outward unit normal and the associated outward normal derivative on $\Gamma$. The symbol $\lesssim$ means an inequality up to some constant $C > 0$.
	
	\subsection{The Schl\"{o}gl model in $d\in \{1,2,3\}$.}	\label{ss_sch}
	Here we consider the  optimal stabilization problem for the Schl\"{o}gl model also known as Nagumo model under control constraints in dimension $d\in \{1,2,3\}$. In the earlier paper \cite{BK:2019}, where the initial conditions where taken in $L^2(\Omega)$ the dimension was restricted to $d=1$. Here we treat 	
	\begin{equation}\label{PSch}
	\V(y_0) = \inf_{\bem y \in \Wiay \\ u \in \Uad \eem} \frac{1}{2} \int_{0}^{\infty} \norm{y(t)}^2_Y dt+ \frac{\alpha}{2} \int_{0}^{\infty} \norm{u(t)}^2_{\calU}dt,
	\end{equation}
	\nin subject to the semilinear parabolic equation		
	\begin{empheq}[left=\empheqlbrace]{align}
	y_t &= \Delta y + R(y) + B u  &\text{ in } Q \\
	\partial_{\nu} y &= 0 &\text{ on } \Sigma \\
	y(x,0) &= y_0 \quad &\text{ in } \Omega.
	\end{empheq}
	\nin where $R$ is the cubic polynomial of the form,
	\begin{equation*}
	R(y) = ay(y - \xi_1)(y - \xi_2), \text{with  real numbers } \xi_1, \xi_2, \text{and } a < 0,
	\end{equation*}
	and $B\in \mathcal{L}(\mathcal{U},Y)$. Note that $R(y)=ay^3 + by^2+cy$, with $b=a(\xi_1+\xi_2)$ and $c=a\xi_1\xi_2$. The origin of the uncontrolled system  is locally unstable, if $\xi_1 \xi_2<0$ and exponentially stable if $\xi_1 \xi_2>0$.
	To  cast this problem in the framework of Section \ref{Sec-Diff}, we set $V=H^1(\Omega)$, $a(v,w)= (\nabla v,\nabla w) - c(v,w)$ and  associated  operator
	\begin{equation}
	Ay = (\Delta + c \bs{I}) y, \text{ with } \calD(A) = \{y \in H^2(\Omega): \partial_{\nu}y|_{\Gamma} = 0\}
	\end{equation}
	and
	\begin{equation}
	F(y) = ay^3 + by^2 .
	\end{equation}
	
	Clearly \eqref{eq:kk2} is satisfied. Concerning condition  \eqref{asp_1}, if $c<0$ then the semigroup generated by $A$ is exponentially stable. In case $c\ge 0$ feedback stabilization by finite dimensional controllers was analyzed in \cite{RT:1975}, and \cite{KR:2019} for instance. It is left  to the reader to  check that the nonlinearity $F$ is twice differentiable as a mapping $F: \Wiay \to \Liy$. For the sake of illustration, we verify the continuity  of the bilinear form of $F''$ on $\Wiay \times \Wiay$. For this purpose, we take $v_1,v_2 \in \Wiay$ and  $y_1, y_2 \in \Wiay$
	and estimate, setting $\delta y= y_2-y_1$
	\begin{align*}
	\norm{(F''(y_2)- F''(y_1))(v_1,v_2)}^2_{\Liy} &\le 6 \int_{0}^{\infty} \int_{\Omega} \abs{a\,\delta y \, v_1v_2}^2 dxdt, \\
	\le 6 \int_{0}^{\infty} |a| \norm{\delta y}^2_{L^6(\Omega)}\norm{v_1}^2_{L^6(\Omega)}\norm{v_2}^2_{L^6(\Omega)} dt &\le 6 C_1
	\norm{v_1}^2_{\Wiay}\norm{v_2}^2_{\Wiay} \int_{0}^{\infty} \norm{\delta y}^2_{\calD(A)}dt ,
	\end{align*}
	where  $C_1$ depends on the continuous embedding $V \to L^6(\Omega)$, which  holds for $d\le3$, and the continuous injection $W_\infty(\mathcal{D}(A),Y) \to C(I;V)$ is used. We also have $F(0) = F'(0) = 0$  and thus  \eqref{eq:fprime} is satisfied. The  following Lemma  together with the compact embedding $ W(0,T;\mathcal{D}(A),Y)\to L^2(0,T;V)$, see e.g. \cite[Satz 8.1.12, p213]{EE:2004}, implies that \eqref{asp_3} is satisfied with $\mathcal{H}=Y$.

	\begin{lemma}\label{lem_dif_z}
		For $y_1, y_2 \in W(0,T;\mathcal{D}(A),Y)$ and $z \in L^{\infty}(0,T;Y)$ the following estimates hold:
		\begin{align}
		\int_{0}^{T} | \ip{y^3_1 - y^3_2}{z}_Y| \ dt & \nonumber\\
		\leq {C_2} & \norm{z}_{L^2(0,T;Y)} \norm{y_1 - y_2}_{L^2(0,T;V)}  \left[\norm{y_1}^2_{W(0,T;\mathcal{D}(A),Y)} + \norm{y_2}^2_{W(0,T;\mathcal{D}(A),Y)}\right],\\
		\int_{0}^{T}| \ip{y^2_1 - y^2_2}{z}_Y| \ dt &\leq {C_3} \norm{z}_{L^2(0,T;Y)}\norm{y_1 - y_2}_{L^2(0,T;V)}\left[\norm{y_1}_{L^2(0,T;V)} + \norm{y_2}_{L^2(0,T;V)}\right],
		\end{align}
		where $C_2, C_3 > 0$ are independent of $y_1,y_2$, and $z$.
	\end{lemma}
	\begin{proof}For the first inequality, we estimate, using embedding constants $C_i$ independent of $y_1,y_2,z$
		\begin{align*}
		\int_{0}^{T} |\ip{y^3_1 - y^3_2}{z}_Y |\ dt &= \int_{0}^{T}\int_{\Omega} |(y^3_1 - y^3_2)z| \ dxdt\\
		&\leq C_4 \int_{0}^{T} \norm{z}_{Y} \norm{y_1 - y_2}_{L^6(\Omega)} \norm{y^2_1 + y_1y_2 + y^2_2}_{L^3(\Omega)} dt\\
		&\leq C_5 \norm{z}_{L^{\infty}(0,T;Y)} \norm{y_1 - y_2}_{L^2(0,T;V)} \norm{y^2_1 + y^2_2}_{L^2(0,T;L^3(\Omega))}.
		\end{align*}
		\nin Then we estimate, for $i = 1,2$,
		\begin{multline*}
		\norm{y^2_i}_{L^2(0,T;L^3(\Omega))} = \left[\int_{0}^{T} \norm{y^2_i}^2_{L^3(\Omega)} dt\right]^{\rfrac{1}{2}} = \left[\int_{0}^{T} \left(\int_{\Omega} \abs{y_i}^6 dx\right)^{\rfrac{2}{3}}dt\right]^{\rfrac{1}{2}} = \left[\int_{0}^{T} \norm{y_i}^4_{L^6(\Omega)}dt\right]^{\rfrac{1}{2}} \\\leq C_6 \norm{y_i}_{C(0,T;V)}\norm{y_i}_{L^2(0,T;V)} \leq C_7 \norm{y_i}^2_{W(0,T;\mathcal{D}(A),Y)}.
		\end{multline*}
		\nin This proves the first inequality. The verification of the second inequality is left to the reader.
	\end{proof}
	\nin Now we turn to \eqref{asp_4} and  verify that $F'(y) = 3 a y^2 + 2b y\in \calL \left( L^2(I;V), L^2(I,Y)\right) $ . We concentrate on the term $y^2$ and estimate for $z \in L^2(I;V)$:
	\begin{multline*}
	\int^\infty_0 \int_\Omega |y|^4 |z|^2 dx\,dt \le \int^\infty_0 \left(\int_\Omega|y|^6 dx \right)^{\rfrac{2}{3}}  \left(\int_\Omega|z|^6 dx \right)^{\rfrac{1}{3}}dt \le \int_0^\infty \norm{y}^4_{L^6(\Omega)} \norm{z}^2_{L^6(\Omega) } dt \\ \le C_8\norm{y}^4_{W_\infty(I;\mathcal{D}(A),Y)} \norm{z}^2_{L^2(I;V)},
	\end{multline*}
	and the claim follows. Finally for \eqref{asp_5}, we utilize the fact that $V \subset L^6(\Omega)$ for $d \leq 3$, and $y \in C(I;V)$ for $y_0 \in B_V(\delta_4)$. This implies $F(y) \in C(I;Y) \subset C(I;V')$.
	
	\subsection{Quartic nonlinearity $y^4$ in $d\in \{1, 2\}$.}
	
	\nin In this case we consider a  semilinear parabolic problem with $F(y) = k y^4$ in dimension $d\in\{1,2\}$.  The computations will be carried out for $d = 2$ but $d = 1$ will readily follow. Specifically the controlled system is given as follows:
	\begin{subequations}\label{qt_mdl}
		\begin{empheq}[left=\empheqlbrace]{align}
		y_t&=  \Delta y + k y^4 + B u  &\text{ in } Q \\
		y &= 0 &\text{ on } \Sigma \\
		y(x,0) &= y_0 \quad &\text{ in } \Omega.
		\end{empheq}
	\end{subequations}
	For this model \eqref{asp_1} is satisfied  with $A$ the Laplacian $Y = L^2(\Omega)$ with Dirichlet boundary conditions. It generates an asymptotically stable analytic semigroup. We use Gagliardo's inequality \cite[p173]{BF:2013} in dimension two to show that $F:{W_\infty(I;\mathcal{D}(A),Y)}\to L^2(I;Y)$ with $F(y)=ky^4$ is well-defined:
	\begin{align*}
	\norm{y^4}^2_{\Liy}  = \int_{0}^{\infty} \norm{y}^8_{L^8(\Omega)}dt \lesssim \int_{0}^{\infty} \left[ \norm{y}^{\rfrac{1}{4}}_{Y}\norm{y}^{\rfrac{3}{4}}_{V} \right]^8 dt \leq C \norm{y}^6_{\Wiay}\norm{y}^2_{\Liy}.
	\end{align*}
	\nin Moreover, assumption  \eqref{asp_2} requires us to show that $F$ is twice continuously differentiable. It can be checked that $F$ is twice continuously differentiable with derivatives given by
	$F'(y) = 4ky^3$ and $F''(y) = 12ky^2$.
	By computations as carried out in Lemma \ref{lem_dif_z}, one can deduce $F'$ and $F''$ are bounded on bounded subsets of $\Wiay$.\\
	
	\nin To verify \eqref{asp_3}, we take $z \in L^{\infty}(0,T;Y)$ and estimate for $y_1,y_2 \in  W_\infty(I;\mathcal{D}(A),Y)$:
	\begin{align*}
	\int_{0}^{T} |\ip{y_1^4 - y_2^4}{z}| dt &\lesssim \int_{0}^{T} \int_{\Omega} |(y_1 - y_2)\left( y_1^3 + y_2^3 \right)z| \ dxdt \lesssim \int_{0}^{\infty} \norm{y_1 - y_2}_{L^4(\Omega)} \norm{y_1^3 + y_2^3}_{L^4(\Omega)}\norm{z}_Y dt,\\
	&\lesssim \norm{z}_{L^{\infty}(0,T;Y)} \int_{0}^{T} \norm{y_1 - y_2}_{L^4(\Omega)}\left( \norm{y_1}^3_{L^{12}(\Omega)} + \norm{y_2}^3_{L^{12}(\Omega)} \right)dt,\\
	&\lesssim \norm{z}_{L^{\infty}(0,T;Y)} \int_{0}^{T} \norm{y_1 - y_2}_V \left( \norm{y_1}^3_{V} + \norm{y_2}^3_{V} \right)dt,\\
	\lesssim \norm{z}_{L^{\infty}(0,T;Y)} & \norm{y_1 - y_2}_{L^2(0,T;V)} \left( \norm{y_1}^2_{C(0,T;V)} + \norm{y_2}^2_{C(0,T;V)} \right) \left( \norm{y_1}_{L^2(0,T;V)} + \norm{y_2}_{L^2(0,T;V)} \right).
	\end{align*}
	This implies \eqref{asp_3}, since weak convergence in $W_\infty(I;\mathcal{D}(A),Y)$ implies strong convergence in $L^2(0,T;V)$. For \eqref{asp_4}, we show $F'(y) = 4 k y^3 \in \calL{(\LiV, L^2(I;Y))}$ for $y \in \Wiay$. We estimate for $z \in \LiV$,
	\begin{equation*}
	\norm{F'(y)z}^2_{\Liy} \lesssim \int_{0}^{\infty} \int_{\Omega} \abs{y^3 z}^2 dxdt \lesssim \int_{0}^{\infty} \norm{y}^6_{L^8(\Omega)} \norm{z}^2_{L^8(\Omega)}dt \leq C \norm{y}^6_{\Wiay} \norm{z}^2_{L^2(I;V)}.
	\end{equation*}
	Assumption \eqref{asp_5} follows by an analogous  argumentation as presented in Section \ref{ss_sch}.\\
	
	\nin With similar arguments the quintic nonlinearity can be considered in dimension 1.

	\subsection{Nonlinearities induced by functions with globally Lipschitz continuous second derivative.}
	

	Consider the system (\hyperref[Pq]{$\calP$}) with $A$ associated to a strongly elliptic second order operator with domain $H^2(\Omega) \cap H^1_0(\Omega)$, so that \eqref{asp_1} are satisfied. Let $F: \Wiay \to \Liy$ be the Nemytskii operator associated to a mapping $\mathfrak f: \BR \to \BR$ which is assumed to be $C^2(\BR)$ with first and second derivatives globally Lipschitz continuous.  We discuss assumption \eqref{asp_2}-\eqref{asp_5} for such an  $F$, and show that they are satisfied for dimensions $d\in \{1, 2, 3\}$. By direct calculation it can be checked that  $F$ is continuously Fr\'{e}chet differentiable for $d \in \{ 1,2,3 \}$. We leave this part to the reader and immediately turn to the second derivative.
	For $y, h_1, h_2 \in \Wiay $  the relevant  expression is given by
	\begin{multline*}
	\norm{F'(y + h_2)h_1 - F'(y)h_1 - F''(y)(h_1,h_2)}^2_{\Liy}\\
	\quad = \int_{0}^{\infty} \int_{\Omega} \abs{(\mathfrak{f}'(y(t,x) + h_2(t,x)) - \mathfrak{f}'(y(t,x)) - \mathfrak{f}''(y(t,x))h_2(t,x)) h_1(t,x)}^2 dxdt \\
	\quad = \int_{0}^{\infty} \int_{\Omega} \abs{g(t,x) \ h_2(t,x)h_1(t,x)}^2 dxdt,
	\end{multline*}
	
	\nin where  $\ds g(t,x) = \int_0^1 (\mathfrak{f}''(y(t,x) + sh_2(t,x)) - \mathfrak{f}''(y(t,x)))ds$.
	Let us denote the global Lipschitz constant of  $\mathfrak{f}''$ by $L$. Then  we estimate
	\begin{align*}
	&\int_{0}^{\infty} \int_{\Omega} \abs{gh_1h_2}^2 dxdt \leq
	\frac{L^2}{4}\int^\infty_0  \norm{h_2(t)}^4_{L^6(\Omega)} \norm{h_1(t)}_{L^6(\Omega)}^2 dt, \\
	&\le \frac{L^2}{4}\int^\infty_0  \norm{h_2(t)}^4_{V} \norm{h_1(t)}_{V}^2 dt
	\leq C \norm{h_2}_{\Wiay}^4  \int^\infty_0 \norm{h_1(t)}_{V}^2 dt\\
	&\leq \tilde C \norm{h_2}^4_{\Wiay} \norm{h_1}^2_{\Wiay}.
	\end{align*}
	This implies that $F$ admits a second Fr\'{e}chet derivative which is bounded on bounded subsets of $\Wiay$.
	Its continuity with respect to $y$ can be checked with similar arguments.
	\nin In order to verify \eqref{asp_3}, we set $\mathcal{H}=Y=L^2(\Omega)$ and consider a sequence
	$y_n \rightharpoonup \hat{y}$ in $W(0,T; \mathcal{D}(A),Y)$ and let $z \in L^{\infty}(0,T;\calH) \subset L^2(0,T;Y)$ be given. Then we estimate
	\begin{equation*}
	\int_{0}^T |\ip{F(y_n) - F(\hat{y})}{z}_{\calH',\calH}| dt = \int_{0}^T \int_{\Omega} |(\mathfrak{f}(y_n) - \mathfrak{f}(\hat{y}))z| \ dxdt \leq C \norm{y_n - \hat{y}}_{L^2(0,T;Y)}\norm{z}_{L^2(0,T;Y)}.
	\end{equation*}
	\nin Then by the compactness of $W(0,T; \mathcal{D}(A),Y)$ in $L^2(0,T;Y) $, we obtain \eqref{asp_3}.\\
	To verify \eqref{asp_4} we proceed with $y \in \Wiay,\ z \in L^2(I;V)$ and estimate
	\begin{eqnarray}
	\norm{F'(y)z}^2_{\Liy} = \int_{0}^{\infty}\int_{\Omega} \left(\mathfrak{f}'(y) - \mathfrak{f}'(0)\right)^2z^2 \ dxdt \leq C\norm{y}^2_{\Wiay}\norm{z}^2_{L^2(I;V)}. \nonumber
	\end{eqnarray}
	\nin This shows $F'(y)$ satisfies \eqref{asp_4}. Assumption \eqref{asp_5} can be verified since $\mathfrak{f}'(y)$, is assumed to be globally Lipschitz continuous  and $y \in C(I;V)$ for $y \in B_V(\delta_4)$.

	\/*
	\appendix		
	\section{Appendix}
	\begin{thm}\cite[p 173]{BF:2013}
		Let $\Omega$ be a Lipschitz domain in $\BR^d$ with compact boundary. Let $p \in [1,\infty]$ and $\ds q \in \left[p, \frac{pd}{d - p} \right]$. There exists a $C > 0$ such that
		\begin{equation}\label{int-inq}
		\norm{\varphi}_{L^q(\Omega)} \leq C \norm{\varphi}_{L^p(\Omega)}^{1 + \rfrac{d}{q} - \rfrac{d}{p}} \norm{\varphi}_{W^{1,p}(\Omega)}^{\rfrac{d}{p} - \rfrac{d}{q}}, \quad \text{ for all } \varphi \in W^{1,p}(\Omega).
		\end{equation}
	\end{thm}
	
	\begin{thm}\cite[p328]{HT:1980}
		Let $\Omega$ be an arbitrary bounded domain, dim $\Omega = d$. Let $0 \leq t \leq s < \infty$ and $ \infty q \geq \wti{q} > 1$. Then, the following embedding holds true:
		\begin{equation}
		W^{s,\wti{q}}(\Omega) \subset W^{t, q}(\Omega), \ s - \frac{d}{\wti{q}} \geq t - \frac{d}{q} \quad \qedsymbol
		\end{equation}
	\end{thm}
	
	\section{Appendix}
	\tcr{Decide $A'$ or $A^*$}\\
	
	\nin We need following lemmas to establish $p \in \Wiay$. The proofs will follow from the same argumentation as in \cite{BKP:2019}
	\begin{lemma}
		Let $\wti{A} \in \calL(\calD(A),Y)$ generate an analytic and exponentially stable semigroup on $Y$ and assume that there exists $M \geq 0$ such that for every $f \in \Liy$ there exists a unique $y \in \Wiay$ satisfying
		\begin{equation*}
		y_t = \wti{A}y + f \text{ on } [0,\infty), \ y(0) = 0, \ \norm{y}_{\Wiay} \leq M \norm{f}_{\Liy}.
		\end{equation*}
		Then there exists $\wti{M}$ such that for all $\Phi \in L^2(I;[\calD(A)]')$ there exists a unique $r \in \Wiya$ such that
		\begin{equation*}
		-r_t = \wti{A}^*r + \Phi, \quad \norm{r}_{\Wiya} \leq \wti{M} \norm{\Phi}_{L^2(I;[\calD(A)]')}.
		\end{equation*}
	\end{lemma}
	\begin{proof}
		\nin Step 1: Let us define $T:\WOiay \to \Liy$ by $Ty = y_t - \wti{A}y$. The adjoint of $T$ is given by $T^*: \Liy \to (\WOiay)'$. In the computation below, we omit spaces from the first duality paring to maintain the readability. Then we compute,
		\begin{equation*}
		\ip{T^*\varphi}{y}_{(\WOiay)',\WOiay} := \ipp{\varphi}{Ty}_{\Liy} = \ipp{\varphi}{y_t - \wti{A}y}_{\Liy}.
		\end{equation*}
		\nin Since by assumption, $T$ is homeomorphic and in particular surjective and injective. Then by the close range theorem, there exists $C > 0$,
		\begin{equation}\label{bb_Ts}
		\norm{\varphi}_{\Liy} \leq C \norm{T^*\varphi}_{(\WOiay)'}, \quad \forall \varphi \in \Liy.
		\end{equation}
		\nin Step 2: Let $\Phi \in L^2(I;[\calD(A)]')$ be arbitrary. Then there exists a unique $r \in \Liy$ such that $T^*r = \Phi$, and by \eqref{bb_Ts} we have
		\begin{equation}\label{r_phi}
		\norm{r}_{\Liy} \leq C \norm{\Phi}_{(\WOiay)'} \leq C \norm{\Phi}_{\LiAd}.
		\end{equation}
		\nin Since $T^*r = \Phi$, we have for all $y \in \Wiay$,
		\begin{align*}
		\ip{\Phi}{y}_{\LiAd, \LiA} &= \ipp{T^*r}{y}_{\Liy} = \ipp{r}{Ty}_{\Liy}\\
		&= \ipp{r}{y_t}_{\Liy} - \ipp{\wti{A}^*r}{y}_{\LiAd, \LiA}.
		\end{align*}
		\nin This implies that the time derivative of $r$, in the sense of distributions, can be extended to a linear form on $\WOiay$ with the formula
		\begin{multline}\label{r_t}
		\ip{r_t}{y}_{(\WOiay)', \WOiay} = - \ipp{r}{y_t}_{\Liy} = - \ip{\Phi + \wti{A}^*r}{y}_{\LiAd, \LiA},\\ \quad y \in \WOiay.
		\end{multline}
		\nin We estimate
		\begin{align*}
		\abs{\ip{r_t}{y}} &\leq \abs{\ip{\Phi + \wti{A}^*r}{y}_{\LiAd, \LiA}} \leq \norm{\Phi + \wti{A}^*r}_{\LiAd}\norm{y}_{\LiA}\\
		&\leq \norm{\Phi}_{\LiAd}\norm{y}_{\LiA} + C \norm{\Phi}_{\LiAd}\norm{y}_{\LiA}, \quad \text{(by \eqref{r_phi})}\\
		&\leq (1 + C)\norm{\Phi}_{\LiAd}\norm{y}_{\LiA}.
		\end{align*}
		\nin Then with \eqref{r_t} and recalling $y \in \WOiay$ is dense in $\LiA$, we obtain that $r_t$ can be extended to a bounded linear form on $\LiA$, i.e. $r_t$ can be extended to an element in $\LiAd$, moreover
		\begin{equation*}
		\norm{r_t}_{\LiAd} \leq (1 + C) \norm{\Phi}_{L^2(I;[\calD(A)]')}.
		\end{equation*}
		\nin It follows that $r \in \Wiay$. Moreover,
		\begin{equation*}
		\norm{r}_{\Wiya} \leq 2(1 + C) \norm{\Phi}_{L^2(I;[\calD(A)]')} \text{ and } -r_t - \wti{A}^*r = \Phi
		\text{ in } \LiAd.
		\end{equation*}
	\end{proof}	
	
	\begin{clr}
		For all $\Phi \in \Liy$, the system
		\begin{equation*}
		-r_t = A_{\rho}^*r + \Phi
		\end{equation*}
		has a unique solution $r \in \Wiay$. Moreover, there exists a constant $M_{\rho} > 0$ independent of $\Phi$ such that
		\begin{equation}
		\norm{r}_{\Wiay} \leq M_{\rho} \norm{\Phi}_{\Liy}.
		\end{equation}
	\end{clr}	
	
	\begin{proof}
		We apply the previous Lemma to $z_t = A_{\rho}^* z + \Psi$ with $\Psi = \left( -A_{\rho}^* \right)^{\rfrac{1}{2}}\Phi \in \Liy$, to obtain $\ds \norm{z}_{\Wiya} \leq \wti M \norm{\Phi}_{}$
	\end{proof}
	\nin We now show that for small initial data $y_0$ is more regular than $p \in \Liy$. For this we need more smoothness of the boundary $\Gamma$. \tcr{adjust $\delta$!}
	
	\begin{prop}\label{prop-p-W-A}
		There exists $\delta > 0$ such that for all $y_0 \in B_V(\delta)$, for all solutions $(\bar{y}, \bar{u})$ of \eqref{P}, there exists a unique costate $p \in \Wiay$ satisfying
		\begin{equation}
		-p_t - A^*p + F'(\bar y)^*p = \bar y \quad \text{ in } L^2(I;Y)
		\end{equation}
	\end{prop}
	
	\begin{proof}
		to be completed.
	\end{proof}
	*/

	\medskip
	Received xxxx 20xx; revised xxxx 20xx.
	\medskip
	
\end{document}